\documentclass{article}

\usepackage{geometry}
\usepackage{hyperref}
\usepackage{amsmath,amssymb,amsthm}
\usepackage{mathtools}

\usepackage{fixltx2e}
\usepackage{hyperref}

\usepackage[active]{srcltx}
\usepackage{graphicx}

\usepackage{amstext} 
\usepackage{array}

\newcommand{\dkr}{\mathbf{d}}

\newcommand{\C}{\mathcal{C} }

\newcommand{\BB}{\mathcal{B} }

\newcommand{\U}{\mathcal{U} }
\newcommand{\V}{\mathcal{V} }

\newcommand{\E}{\mathbb{E}^x }
\newcommand{\Rset}{\mathbb{R}}
\newcommand{\RsetN}{\Rset^d}

\newcommand{\Prob}{\mathcal{P}}

\renewcommand{\div}{\mathrm{div}}

\newcommand{\ds}{\displaystyle}

\newtheorem{thm}{\textbf{Theorem}}
\newtheorem{lem}[thm]{\textbf{Lemma}}
\newtheorem{prop}[thm]{\textbf{Proposition}}
\theoremstyle{remark}
\newtheorem{rem}[thm]{\textbf{Remark}}

\newtheorem{exe}[thm]{\textbf{Example}}
\theoremstyle{definition}

\newtheoremstyle{Claim}{}{}{\itshape}{}{\itshape\bfseries}{:}{ }{#1}
\theoremstyle{Claim}

\newcolumntype{C}{>{$}c<{$}}

\begin{document}
\nonstopmode
\markboth{Achdou, Bardi, Cirant}{Mean Field Games models of segregation}

\title{Mean Field Games models of segregation}

\author{Yves Achdou, Martino Bardi and Marco Cirant }

\maketitle

\begin{abstract}
This paper introduces and analyses some models in the framework of Mean Field Games 
 describing  interactions 
 between two   populations motivated by the studies on  urban settlements and 
 residential choice by Thomas Schelling. For static games, a large population limit is proved. For differential games with noise, the
 existence of solutions is established for the systems of partial differential equations of Mean Field Game theory, in the stationary and in the evolutive case. Numerical methods are proposed, with several simulations. In the examples and in the numerical results,
particular emphasis is put on the phenomenon of segregation between the populations.

\end{abstract}

\noindent
{\footnotesize \textbf{AMS-Subject Classification}}. {\footnotesize 91A13, 49N70, 35K55}\\
{\footnotesize \textbf{Keywords}}. {\footnotesize Mean Field Games, multi-populations models, large populations limit, systems of parabolic equations, finite difference methods, segregation.}

\tableofcontents


\section{Introduction} 

The theory of Mean Field Games (MFG, in short)   is a branch of 
Dynamic Games which aims at modeling and analyzing complex decision processes involving a large number of indistinguishable rational agents  who have individually a very small influence on the overall system and are, on the other hand, influenced by the distribution of the other agents.
%
It originated about ten years ago  in the independent work of J. M. Lasry and P.L. Lions, 
Ref. \cite{LasryLions}, and of M.Y. Huang, P. E. Caines and R. Malham{\'e} 
Refs. \cite{HCM:06}, \cite{HCM:07ieee}. In the case of independent noises affecting the agents, the main equations describing MFG are a  Hamilton-Jacobi-Bellman parabolic equation for the value function of the representative agent coupled with a Kolmogorov-Fokker-Planck equation for the density of the population, 
the former backward in time with a terminal condition and the latter forward in time with an initial condition. 
Recently the theory and applications of MFG have been growing very fast: we refer to P.-L. Lions' courses on the site of the  Coll{\`e}ge de France http://www.college-de-france.fr/site/en-pierre-louis-lions/, the lecture notes Refs. \cite{Gueantnotes} and \cite{CardaliaguetNotes}, and the books Refs. \cite{BFY}, \cite{GomesBookEcon}, and \cite{GomesBook}. A major recent breakthrough by Cardaliaguet, Delarue, Lasry, and Lions is the solution of a PDE in the space of probablilty measures, called master equation, which describes MFGs with a common noise affecting all players and allows to prove general convergence results of $N$-person differential games to a MFG as $N\to\infty$,  in a suitable sense.

 The goal of this paper is to propose some models in the framework of Mean Field Games to describe some kinds of interactions 
 between two different populations, each formed by a large number of indistinguishable agents. Such phenomena arise, for instance, in urban settlements, ecosystems, pedestrian dynamics, see, e.g., Refs. \cite{CPT}, \cite{BeGi}, and the references therein. We will focus in particular on models of residential choice possibly leading to  segregated neighborhoods. We are inspired by the pioneering work of the Nobel Prize in Economics Thomas Schelling, Refs. \cite{S71}, \cite{S78}, and some of its developments until recently, see, e.g., Refs. \cite{Zh:04a}, \cite{BruMa}, \cite{BarTas}, \cite{Zh:11}, \cite{GGJ2}, the survey \cite{GGJ}, and the references therein. 
However, different from the sociologic and economic literature where the models are usually discrete in space and time, we propose games continuous in space and either static, for which we derive rigorously the large population limit, or in continuous time, with  the dynamics of each player described by a controlled system affected by noise. In the differential game, the preferences of the players are described by a 
 cost functional integrated in time that each players seeks to minimise. We consider finite horizon problems as well as games with long-time average cost (also 
  called ergodic cost).
 
 Our 
  analytic results are on the existence of solutions to the system of the four PDEs associated to the two-population MFG, with Neumann boundary conditions modelling the boundedness of the city where the agents live. The PDEs are elliptic in the case of ergodic cost, with an additive eigenvalue in each of the two H-J-B equations; the case of several populations was treated by the second author and Feleqi with periodic boundary conditions (i.e., the state space of the agents is a torus, Refs. \cite{BF}, \cite{Feleqi}), and by the third author with Neumann boundary conditions, Ref. \cite{Cirant}. For finite horizon costs, the PDEs 
 are parabolic (two backward and two forward in time) and existence is known for a single population and periodic boundary conditions; we extend it to two populations and Neumann conditions. 
 Uniqueness of solutions holds for a single population under a restrictive monotonicity condition (Ref. \cite{LasryLions}) and  is not expected to hold for several populations. In fact, we provide examples of non uniqueness by showing that the same game can have segregated solutions as well as non-segregated ones, such as uniform distributions of both populations.
 
 One of the most interesting issues about these models is the qualitative behavior of solutions, in particular whether two initially mixed population tend to segregate, i.e., to concentrate in different parts of the city.
 Schelling's most striking discovery was that very moderate preferences for
same-population neighbors at the individual level can lead to complete residential
segregation at the macro level. 
For example, if every agent requires at least half of her
neighbors to belong to the same population, and moves only if the percentage is below this threshold,
the final outcome,
after a sequence of moves, is almost always complete segregation. Nowadays several softwares freely available on the internet allow such simulations and show that segregation eventually occurs, with random initial conditions, even with much milder thresholds, i.e., lower than $1/2$, see,  e.g., NetLogo (http://ccl.northwestern.edu/netlogo/).
Thus Schelling's conclusion was 
that the ``macrobehavior'' in a society may not reflect the ``micromotives''  of its individual
members (Ref. \cite{S78}). His early experiments are considered today among the first prototypes of artificial societies, see, e.g., Ref. \cite{Rauch}.

We study the qualitative behavior of solutions by numerical methods. We use the techniques 
 introduced in MFG with a single population and periodic boundary conditions by the first author, Capuzzo Dolcetta, and Camilli in Refs. \cite{AchdouCapuzzo}, \cite{AchdouPlanning}, and \cite{AchdouCamilli}. We present finite difference schemes for the stationary PDEs associated to ergodic costs as well as for the evolutive backward-forward system of the finite horizon problem. For both cases we show that segregation occurs with low preference thresholds, so Schelling's principle is valid also in our MFG models. We also compare the results for different thresholds, showing that a higher threshold pushes a population to concentrate in a smaller space, and we also observe the instability arising if both populations are rather xenophobic, leading to oscillations in time. 
 Finally we present a 2-d example of pedestrian dynamics with two populations.

More references to the literature on MFG will be given throughout the paper.

The paper is organised as follows. In Section \ref{Schelling}, we propose several forms of cost functionals that reflect the preferences described by Schelling, 
 with variants and generalizations. In Section \ref{staticMF}, we prove a large population limit for the static game, following the method of Lions and Cardaliaguet, Ref. \cite{CardaliaguetNotes}, and give some simple examples of Mean Field equilibria. In Section \ref{differential}, we first introduce a dynamics driven by a stochastic control system, the long-time average cost, and the stationary MFG PDEs associated to them, followed by an example of coexistence of segregated and non-segregated solutions. Then we describe the finite horizon problem, the evolutive MFG PDEs for it, and prove an existence theorem. Section \ref{sec_numstrategies} illustrates the numerical methods for the MFG PDEs. The final Section \ref{sec_numresults} contains several simulations for the stationary and evolutive cases, in 1 and 2 dimensions.

\section{
Static games in continuous space 
 inspired by T. Schelling} 
 \label{Schelling}

In this section, we propose a class of static (one-shot) games with two populations of players whose positions are taken in a bounded set  $\overline{\Omega} \subset \RsetN$. Within each population, all players have the same cost functional to minimize. We choose 
such functionals in a way that reproduces the main features of 
 the classical models of segregated neighborhoods by 
  Schelling, Ref. \cite{S71} and 
 \cite{S78}, 
and of some of their subsequent developments.  We fix a neighborhood $\U(x)$ for each point  $x\in \overline{\Omega}$ and consider the amount of each population living in such neighborhood, $N_1(x), N_2(x)$. 
In the simplest models, 
the utility $U_k$ (= minus the cost) of an individual of the $k$-th species living at the position $x$ depends only on the quantity 
\begin{equation}
\label{s_i}
s_k:=\frac{N_k(x)}{N_1(x)+N_2(x)}
\end{equation}
 and has the shape shown in Figure \ref{utility_plot}, that is,
\begin{equation}
\label{utility}
U_k(s_k):=\left\{
\begin{array}{ll}
\theta_k(s_k-a_k) \qquad & \qquad\text{if } s_k<a_k ,\\
0 & \qquad\text{else,}
\end{array}
\right.
\end{equation}
where $\theta_k>0$ and $0\leq a_k\leq 1$. Here $s_k$ is the percentage of population $k$ living in $\U(x)$ and $a_k$ is a threshold of happiness: if $s_k$ is below it the player of the $k$-th species  at the position $x$ has a negative utility, i.e., a positive cost. 

\begin{figure}
\centering
\caption{\footnotesize The utility function $U_k
$ ($\theta_k
=2$, $a_k
=0.4$).}\label{utility_plot}
    \includegraphics[width=6cm]{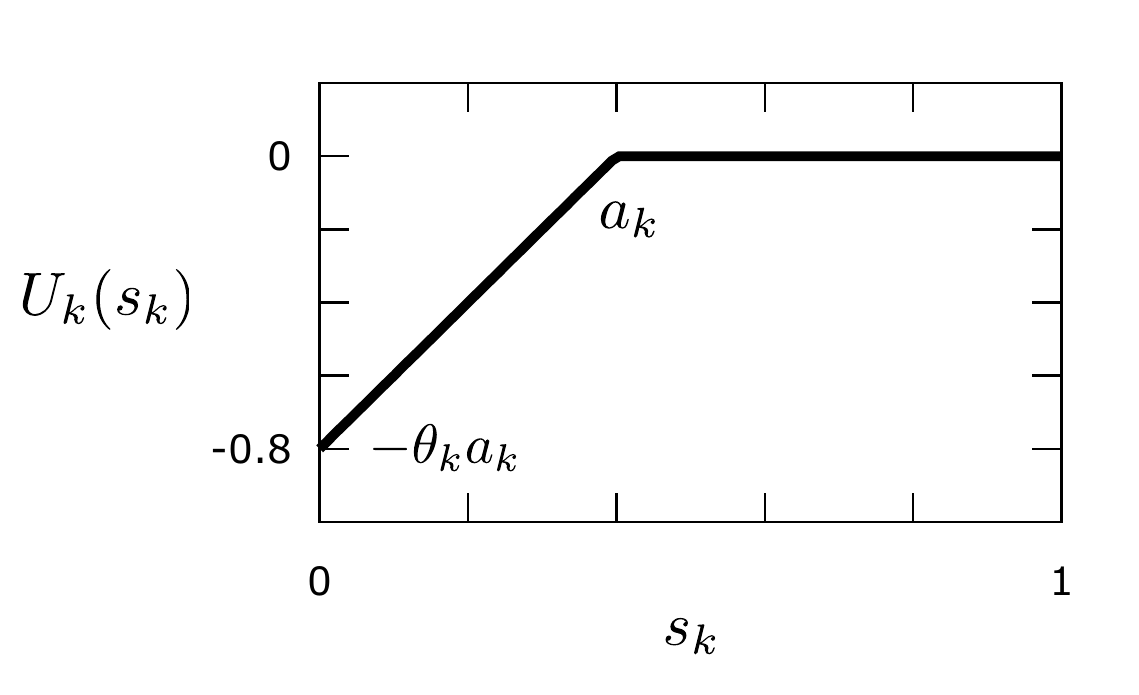}
\end{figure}

In the Schelling's model and in the differential games of Section \ref{differential}, the agent then 
 moves and looks for a location with a higher 
  value of $s_k$, possibly $s_k>a_k$. In the static games of this section, we look for equilibrium distributions of the players that are Nash equilibria for the game of minimizing the individual costs.
    In most of the recent literature the parameter $a_k$ is taken to be 1/2 for both populations, but in Schelling's original examples it is often below 1/2, therefore modeling populations that are not xenophobic and that just do not want that their own group be too small in their neighborhood. 

  The shape of the utility function \eqref{utility} is ``peaked at $a_k$'', as one of those considered in Ref. \cite{BarTas} and a limit case of those in Ref. \cite{Zh:04a}, \cite{Zh:11}; 
  for the slope $\theta_k$ very large it approximates the stair-like utility of Schelling, and for $a_k=1$ it is the linear utility of Ref. \cite{BruMa}; see Ref. \cite{GGJ} for a survey. References \cite{Zh:04a}, \cite{Zh:11}, and \cite{BarTas} consider also utilities decreasing on the right of   $a_k$: although we do not consider these cases in the numerical simulations, they satisfy the same boundedness conditions as our models and therefore fit into our analysis of Sections \ref{staticMF}, \ref{differential}, and \ref{sec_numstrategies}.

  We will consider also more general cost functionals that depend on $N_1(x)$ and $N_2(x)$ separately, not only via $s_k$, and definitions of $N_1(x), N_2(x)$ as measures of the number of individuals weighted by the distance from $x$.
Our assumptions will be general enough to include examples in fields different from residential segregation, such as crowd motion and pedestrian dynamics, see Ref. \cite{CPT} for a general presentation and  Ref. \cite{Lach} and \cite{LW} for Mean-Field Games models with two populations.

\subsection{A basic game with two populations of $
N
$ players.}
\label{basic}

We consider a one-shot 
 game with $2N$ players divided in two populations. The vector $(x_1, \ldots, x_N)$ represents the positions of the players of the first population and $(y_1, \ldots, y_N)$ those of  the players of the second one, where  $x_i, y_i\in \overline{\Omega}$ 
and  $
{\Omega} \subset \RsetN$ is an open and bounded set. We adopt the conventions and notations of Mean-Field Games, see Ref. \cite{LasryLions}, and associate to each player a \textit{cost} (instead of a utility) that the player seeks to \textit{minimise} (instead of maximise), it is denoted with $F^{1,N}_i$ for the $i$-th player of the first population and with $F^{2,N}_i$ for the $i$-th player of the second population.
%
The first kind of cost functionals we propose are 
\begin{multline}
\label{F1}
F^{1,N}_i(x_1, \ldots, x_N, y_1, \ldots, y_N) = \\
\theta_1\left( \frac{\sharp \{x_j \in \U(x_i) : j \neq i\} }{\sharp \{x_j \in \U(x_i) : j \neq i\} + \sharp \{y_j \in \U(x_i) \} + \eta_1(N-1) } - a_1 \right)^-,
\end{multline}
where $\theta_1>0, 0\leq a_1 \leq 1, \eta_1 \geq 0$, $\sharp X$ denotes the cardinality of the (finite) set $X$, $\U(x)$ is some neighborhood of $x$ (for example $B_r(x) \cap \overline{\Omega}$, where $B_r(x)$ is the ball centered at $x$ of radius $r$, or $S_r(x) \cap \overline{\Omega}$, where $S_r(x)$ is the square centered at $x$ of side length $r$), and $(t)^-$ denotes the negative part of $t$, i.e., $(t)^-=-t$ if $t<0$ and $(t)^-=0$ if $t\geq 0$.
As before, $a_1 \in [0,1]$ is the ``threshold of happiness'' of any player of the first population: his cost is null if the ratio of the individuals of his own kind in the neighborhood 
 is above this threshold, whereas the cost is positive with slope $\theta_1$ below the threshold. 
Note that, for  $\eta_1 = 0$ and $U_1, s_1$ defined by \eqref{utility}, \eqref{s_i},
  \[
 F^{1,N}_i
  := -U_1(s_1), \quad N_1(x)=\sharp \{x_j \in \U(x_i) : j \neq i\}, \quad N_2(x)=\sharp \{y_j \in \U(x_i) 
 \} .
 \]
 In the following, however, we will assume $\eta_1 > 0$ (and small) in order to avoid the indeterminacy of the ratio   $s_1$ \eqref{s_i} as $N_1(x)+N_2(x)\to 0$. 
 This assumption makes the cost continuous, 
 and it has the following interpretation: suppose that a player is surrounded just by individuals of his own kind, i.e. $\sharp \{y_j \in \U(x_i) \} = 0$, then the cost he pays is null as long as
\[
 N_1(x)=\sharp \{x_j \in \U(x_i) : j \neq i\} \ge \frac{a_1 \eta_1 }{1-a_1} (N-1) .
\]
But if $ N_1(x)$ becomes too small
 he pays a positive cost (tending to $\theta_1a_1$ as $ N_1(x)\to 0$). 
 This means that it is uncomfortable to live in an almost desert neighborhood. 

We introduce the notation
\begin{equation}\label{eqG}
G(r,s; a, t) := \left( \frac{r}{r+s+t} - a \right)^-,
\end{equation}
and observe that $G: [0, +\infty) \times [0, +\infty) \times [0,1] \times (0, 1) \rightarrow [0, +\infty)$
is a continuous and bounded function of 
 $r, s$ for each $a,t$ fixed.
We rewrite
\begin{multline*}
F^{1,N}_i(x_1, \ldots, x_N, y_1, \ldots, y_N) \\ = \theta_1 G(\sharp \{x_j \in \U(x_i) : j \neq i\}, \sharp \{y_j \in \U(x_i) \}; a_1, \eta_1(N-1)),
\end{multline*}
The cost for each player of the second population is
\begin{multline*}
F^{2,N}_i(x_1, \ldots, x_N, y_1, \ldots, y_N) \\ = \theta_2 G(\sharp \{y_j \in \U(y_i) : j \neq i\}, \sharp \{x_j \in \U(y_i) \}; a_2, \eta_2(N-1)),
\end{multline*}
where $a_2 \in [0,1]$ represents the 
threshold of happiness of this population 
 and $\theta_2, \eta_2 >0$. It has the same form as $F^{1,N}_i$, but the three parameters $a_2, \theta_2, \eta_2$ can be different from $a_1, \theta_1, \eta_1$.

We note that 
the costs depend on the position of the players only via the empirical measures of the two populations. As usual in the theory of Mean-Field Games they 
can be generated by maps over probability measures as follows
\begin{equation}\label{eqwithmass}
F^{1,N}_i(x_1, \ldots, x_N, y_1, \ldots, y_N) = V^{1,N}\left[\frac{1}{N-1}\sum_{i \neq j} \delta_{x_j}, \frac{1}{N}\sum \delta_{y_j} \right](x_i),
\end{equation}
where  $V^{1,N}: \Prob(\overline{\Omega}) \times \Prob(\overline{\Omega})
\rightarrow C(\overline{\Omega} )$ is defined by 
\begin{equation}
\label{Vbase1} 
V^{1,N} [m_1, m_2](x
) := \theta_1 G\left((N-1)\int_{\U(x
)} m_1, N \int_{\U(x
)} m_2; a_1, \eta_1(N-1) \right) ,
\end{equation}
where $ \Prob(\overline{\Omega})$ denotes the set of all probability measures over $\overline{\Omega}$.
In the same way,
\begin{multline}
\label{Vbase2} 
F^{2,N}_i(x_1, \ldots, x_N, y_1, \ldots, y_N) = V^{2,N}\left[\frac{1}{N}\sum \delta_{x_j}, \frac{1}{N-1}\sum_{i \neq j} \delta_{y_j}
 \right](y_i) \\
= \theta_2 G\left((N-1)\int_{\U(y_i)} \frac{1}{N-1}\sum_{i \neq j} \delta_{y_j}, N \int_{\U(y_i)} \frac{1}{N}\sum \delta_{x_j}; a_2, \eta_2 (N-1)\right).
\end{multline}

In the rest of the paper, we will assume 
\[
\theta_1=\theta_2=1 .
\]
This is done merely for simplifying the notations, all the results and proofs of the paper remain valid for any positive values of $\theta_k$.

\subsection{Overcrowding and family effects}
\label{family}


In the discrete model of Schelling, there is a structural impossibility of overcrowding: every player occupies a position in a chessboard, and every slot can host at most 
 one player. In our continuous model, there is no constraint on the local density and the individuals 
  may even concentrate at a single point of the domain. In order to avoid this unrealistic phenomenon,
   we shall introduce an 
   {\em overcrowding} 
    term in the costs $F^{k,N}_i$:
\[
\begin{split}
& \hat{F}^{1,N}_i(x_1, \ldots, y_N) = F^{1,N}_i + C_1 [(\sharp \{x_j \in \U(x_i)\} + \sharp \{y_j \in \U(x_i) \})/(2N) - b_1 ]^+, \\
& \hat{F}^{2,N}_i(x_1, \ldots, y_N) = F^{2,N}_i + C_2 [(\sharp \{x_j \in \U(y_i)\} + \sharp \{y_j \in \U(y_i) \})/(2N) - b_2 ]^+,
\end{split}
\]
for every $i=1,\ldots,N$, so every player starts paying a positive cost when the total number of players in his neighborhood overcomes 
 the threshold $b_k 2N$; thus 
  $b_k \ge 0$ represents the maximum percentage of 
   the whole population that is tolerated at no cost. Here $C_k$ are positive constants, possibly large: when the concentration of players is too high in some regions, the discomfort might be due to 
    overcrowding and not necessarily to an unsatisfactory ratio between the total number of individuals of the two populations (the $F^{k,N}_i$ term).

The maps over probability measures that generate these costs are
\[
\begin{split}
& \hat{V}^1[m_1, m_2](x) := V^1[m_1, m_2](x) + C_1 \left[\int_{\U(x)}\frac{m_1 + m_2}2 - b_1 \right]^+, \\
& \hat{V}^2[m_1, m_2](x) := V^2[m_1, m_2](x) + C_2 \left[\int_{\U(x)}\frac{m_1 + m_2}2 
 - b_2 \right]^+
\end{split}
\]
for the two populations.

\smallskip

Next we take into account that an individual 
 may be influenced also by the opinions 
 of other individuals 
 living around him. 
 A first attempt to model this is adding to the cost of each player the costs paid by the players of his own kind  and very close to him, e.g., by his family, leading to 
\[
\begin{split}
& \overline{F}^{1,N}_i(x_1, \ldots, x_N, y_1, \ldots, y_N) = \frac{1}{N}\sum_{l \, : \, x_l \in \V(x_i)} F^{1,N}_l(x_1, \ldots, x_N, y_1, \ldots, y_N),  \\
& \overline{F}^{2,N}_i(x_1, \ldots, x_N, y_1, \ldots, y_N) = \frac{1}{N}\sum_{l \, : \, y_l \in \V(y_i)} F^{2,N}_l(x_1, \ldots, x_N, y_1, \ldots, y_N) ,
\end{split}
\]
where $\V(x)$ is a neighborhood of $x$ in $\Omega$. 
This can be refined by assuming  that the opinion of other neighbors is 
weighted 
 by a function that depends upon the distance from the individual
\begin{equation}\label{costimedi}
\overline{F}^{k,N}_i(x_1, \ldots, x_N, y_1, \ldots, y_N) = \frac{1}{N} \sum_{l=1}^N F^{k,N}_l(x_1, \ldots, x_N, y_1, \ldots, y_N) W(x_i, x_l),
\end{equation}
$k=1,2$, where $W: \overline{\Omega} \times \overline{\Omega} \rightarrow \Rset$ is 
nonnegative and 
such that $W(x_i, \cdot)$ 
has support in $\V(x_i)$.
Hence, combining  \eqref{eqwithmass} and \eqref{Vbase2} with \eqref{costimedi}, we arrive at
\begin{equation}
\overline{F}^{k,N}_i(x_1, \ldots, x_N, y_1, \ldots, y_N) = 
 \int_{\overline{\Omega}} W(x_i, z
) V^{k,N}(z
) 
\frac{1}{N} \sum_{l=1}^N  \delta_{x_l} (dz) 
,
\end{equation}
where 
$$
V^{1,N}(z):=V^{1,N}\left[\frac{1}{N-1}\sum_{i \neq j} \delta_{x_j}, \frac{1}{N}\sum \delta_{y_j} \right](z) , 
$$$$
V^{2,N}(z):=V^{2,N}\left[\frac{1}{N}\sum \delta_{x_j}, \frac{1}{N-1}\sum_{i \neq j} \delta_{y_j}
 \right](z) .
 $$

\subsection{More regular cost functionals}
\label{more}

The cost functionals proposed so far involve the amount of individuals in a neighborhood of $x$ that can be written as 
\[
\int_{\U(x)} dm_k(y) = \int_{\overline{\Omega}} \chi_{\U(x)}(y) dm_k(y),
\]
where $m_k$ is the empirical measure of the $k$-th population and $\chi_{\U(x)}(\cdot)$ is the indicator function of the set $\U(x)$, i.e., $\chi_{\U(x)}(y)=1$ if $y\in \U(x)$ and $\chi_{\U(x)}(y)=0$ otherwise.
It is useful to consider regularized versions of such integrals 
where $\chi_{\U(x)}(x,y)$ is approximated by a nonnegative smooth kernel $K(\cdot,\cdot)$ 
such that  $K(x,y)=1$ if $y\in \U(x)$ and $K(x,y)=0$ for $y$ out of a small neighborhood of $\U(x)$. 
The 
 cost functionals of Section \ref{basic} are modified to 
\begin{multline}
\label {V1reg}
V^{1,N}\left[m_1, m_2 \right](x) :=\\ G\left((N-1)\int_{\overline{\Omega}} K(x,y) dm_1(y), N \int_{\overline{\Omega}} K(x,y) dm_2(y); a_1, \eta_1(N-1) \right),
\end{multline}
\begin{multline}
\label {V2reg}
V^{2,N}\left[m_1, m_2 \right](x)  :=\\ G\left((N-1) \int_{\overline{\Omega}} K(x,y) dm_2(y), N \int_{\overline{\Omega}} K(x,y) dm_1(y); a_2, \eta_2 (N-1)\right).
\end{multline}
As we will see in the next section, these new functionals are continuous on 
$\Prob(\overline{\Omega})$ endowed with a suitable notion of distance between measures. 

In the present continuous-space setting, they are also more realistic, because individuals near the boundary of  $\U(x)$ still count in the cost but with 
small weights. More generally, $K$ can be a suitable decreasing function of the distance between $x$ and $y$.


\section{Static Mean-Field Games with two populations}
\label{staticMF} 
In this section, we derive a pair of equations in $\Prob(\overline{\Omega})$ that describe the one-shot Mean-Field Game with two populations of players. They are obtained by taking the limit as $N\to \infty$ of Nash equilibria in the game with $N+N$ players. They are the natural extension to two populations of the equation proposed by Lions for a single population in his lectures at the College de France, see Ref. \cite{CardaliaguetNotes}.

In the sequel, we consider $\Prob(\overline{\Omega})$ as a metric space with the Kantorovich-Rubinstein distance\footnote{We recall that $\dkr(\mu, \nu) = \sup\left\{ \int_{\overline{\Omega}} \phi(x) (\mu - \nu) (dx) \, | \, \phi: \overline{\Omega} \to \Rset \text{ is $1$-Lipschitz continuous} \right\}$.} between  two measures $\mu, \nu$ that we denote with $\dkr(\mu, \nu)$, whose topology corresponds to the weak$^*$ convergence of measures (see, e.g., Ref. \cite{CardaliaguetNotes}).

\subsection{The large populations limit}
\label{largepop}
Let $F^{1,N}_1, \ldots, F^{1,N}_N, F^{2,N}_1, \ldots, F^{2,N}_N : \overline{\Omega}^{2N}\to \Rset$ be the cost functions of a 
 game with  two populations of $N$ players each. Suppose that there exist continuous
 $V^1, V^2: \Prob(\overline{\Omega}) \times \Prob(\overline{\Omega})
\rightarrow C(\overline{\Omega} )$ such that, for all $N$ and $i = 1, \ldots, N$,

\begin{align}\label{FV}
F^{1,N}_i(x_1, \ldots, x_N, y_1, \ldots, y_N) = V^{1}\left[\frac{1}{N-1}\sum_{i \neq j} \delta_{x_j}, \frac{1}{N}\sum \delta_{y_j}
 \right](x_i) + o(1) \\ 
 \label{FV2}
F^{2,N}_i(x_1, \ldots, x_N, y_1, \ldots, y_N) = V^{2}\left[\frac{1}{N}\sum \delta_{x_j}, \frac{1}{N-1}\sum_{i \neq j} \delta_{y_j}
 \right](y_i) + o(1),
\end{align}
where $o(1) \rightarrow 0$ as $N \rightarrow \infty$ uniformly with respect to $x_i, y_j$. 

For $(\bar{x}^N_1, \ldots, \bar{x}^N_N, \bar{y}^N_1, \ldots, \bar{y}^N_N)\in\overline{\Omega}^{2N}$, denote the empirical measures with
\[
\bar{m}_1^N := \frac{1}{N}\sum_{j=1}^N \delta_{\bar{x}_j^N}, \quad \bar{m}_2^N := \frac{1}{N}\sum_{j=1}^N \delta_{\bar{y}_j^N}.
\]
The next result is the large population limit of Nash equilibria.
\begin{prop}
\label{nashlimits} 
Assume \eqref{FV}, \eqref{FV2},  and that, for all $N$, $(\bar{x}^N_1, \ldots, \bar{x}^N_N, \bar{y}^N_1, \ldots, \bar{y}^N_N)$ is a Nash equilibrium for the game with cost functions $F^{1,N}_1, \ldots, F^{1,N}_N$, $F^{2,N}_1, \ldots, F^{2,N}_N$. 
Then, up to subsequences, the sequences of measures $(\bar{m}_1^N)$, $(\bar{m}_2^N)$ converge, respectively,  to $\bar{m}_1, \bar{m}_2 \in \Prob(\overline{\Omega})$ such that
\begin{equation}
\label{MFoneshot}
\int_{\overline{\Omega}} V^k[\bar{m}_1, \bar{m}_2](x) d \bar{m}_k(x) = \inf_{\mu \in \Prob(\overline{\Omega})} \int_{\overline{\Omega}} V^k[\bar{m}_1, \bar{m}_2](x) d \mu(x), \quad k=1,2.
\end{equation}
\end{prop}

\begin{proof} By compactness, $\bar{m}_k^N \rightarrow m_k$ as $N \rightarrow \infty$ (up to subsequences); we need to prove that $\bar{m}_k$ satisfy \eqref{MFoneshot}. Let $\epsilon > 0$, for all $N \ge \bar{N}=\bar{N}(\epsilon)$ we have that for all $z \in \overline{\Omega}$, $i = 1, \ldots, N$,
\[
V^{1}\left[\frac{1}{N-1}\sum_{i \neq j} \delta_{\bar{x}^N_j}, \frac{1}{N}\sum \delta_{\bar{y}^N_j}
 \right](\bar{x}^N_i) \le 
V^{1}\left[\frac{1}{N-1}\sum_{i \neq j} \delta_{\bar{x}^N_j}, \frac{1}{N}\sum \delta_{\bar{y}^N_j}
 \right](z) + \epsilon
\]
by definition of Nash equilibrium and \eqref{FV}, so the measure $\delta_{\bar{x}_i^N}$ satisfies for all $\mu \in \Prob(\overline{\Omega})$
\begin{multline*}
\int_{\overline{\Omega}} V^{1}\left[\frac{1}{N-1} \sum_{j \neq i} \delta_{\bar{x}_j^N}, \frac{1}{N} \sum_{j} \delta_{\bar{y}_j^N}\right](x) d \delta_{\bar{x}_i^N} (x) \le \\ 
 \int_{\overline{\Omega}} V^{1}\left[\frac{1}{N-1} \sum_{j \neq i} \delta_{\bar{x}_j^N}, \frac{1}{N} \sum_{j} \delta_{\bar{y}_j^N}\right](x) d \mu(x)+ \epsilon.
\end{multline*}
Since 
$\dkr \left(\frac{1}{N-1} \sum_{j \neq i} \delta_{\bar{x}_j^N}, \bar{m}_1^N \right) \to 0
$, by continuity of $V^1$
\[
\left| V^{1}\left[\frac{1}{N-1} \sum_{j \neq i} \delta_{\bar{x}_j^N}, \frac{1}{N} \sum_{j} \delta_{\bar{y}_j^N}\right](x) - V^1[\bar{m}_1^N, \bar{m}_2^N](x) \right| \le \epsilon
\]
for all $x\in \overline{\Omega}$ and 
$N \ge \bar{N}$, so
\[
\int_{\overline{\Omega}} V^1[\bar{m}_1^N, \bar{m}_2^N](x) d \delta_{\bar{x}_i^N} (x) \le \\ 
 \int_{\overline{\Omega}} V^1[\bar{m}_1^N, \bar{m}_2^N](x) d \mu(x)+ 3\epsilon.
\]
Then we take the sum for $i=1,\dots,N$ and the $\inf_\mu$, divide by $N$ and get
\[
\int_{\overline{\Omega}} V^1[\bar{m}_1^N, \bar{m}_2^N](x) d \bar{m}_1^N (x) \le \\ 
\inf_{\mu \in \Prob(\overline{\Omega})} \int_{\overline{\Omega}} V^1[\bar{m}_1^N, \bar{m}_2^N](x) d \mu(x)+ 3\epsilon.
\]
Using again that continuity of $V^1$, by passing to the limit as $N \rightarrow \infty$ and then $\epsilon\to 0$ we obtain \eqref{MFoneshot} for $k=1$. The argument for $k=2$ is analogous, by using \eqref{FV2} instead of \eqref{FV}.
\end{proof}

\begin{rem} 
The two equations \eqref{MFoneshot} define 
a {\em Mean-Field equilibrium} $(\bar{m}_1, \bar{m}_2)$ for any game with  two populations associated to the functionals $V^1, V^2$. They are easily seen to be equivalent to the equations 
\begin{equation}
\label{MFoneshot2} 
\forall \, x\in \text{supp } \bar m_k \qquad V^k[\bar{m}_1, \bar{m}_2](x) = \min_{z\in \overline{\Omega}} V^k[\bar{m}_1, \bar{m}_2](z) , \quad k=1,2 ,
\end{equation}
see Ref. \cite{CardaliaguetNotes}, Section 2.2, for the case of a single population.
\end{rem}
\begin{rem} 
\label{mixed}
The assumption of existence of a Nash equilibrium for the $N+N$ game in the previous theorem may look restrictive because Nash equilibria may not exist without further assumptions.
However, the classical Nash Theorem guarantees that Nash equilibria exist if we allow players to use \textit{mixed strategies}, i.e., to minimise over elements of $\Prob(\overline{\Omega})$. 
Moreover, all players of the same population use the same cost function, so one can consider Nash equilibria in mixed strategies that are symmetric within each population, as in Section 8 of Ref. \cite{CardaliaguetNotes}. Then one can derive the equations \eqref{MFoneshot} and \eqref{MFoneshot2} via the large population limit by assuming 
$(x,m_1,m_2)\mapsto V^k[m_1,m_2](x)$ both Lipschitz continuous, but not the existence of a Nash equilibrium in pure strategies, following Section 2.3 of
Ref. \cite{CardaliaguetNotes}. 

\end{rem}

\subsection{Examples}
\label{Exa}
Here we show that the models of Section \ref{Schelling} satisfy the assumptions of Proposition \ref{nashlimits} or Remark \ref{mixed} as soon as the the amount of players in a neighborhood is regularized as in Section \ref{more}. This is based on the next simple result.
\begin{lem}
\label{liplip}
If $K : \Rset^d\times\Rset^d\to \Rset$ is Lipschitz continuous, then the map  $\overline{\Omega}\times\Prob(\overline{\Omega}) \to \Rset^d$, $(x,m) \mapsto \int_{\overline{\Omega}} K(x,y) dm(y)$ is Lipschitz continuous.
\end{lem}
\begin{proof} The Lipschitz continuity in $x$ is immediate. For the Lipschitz continuity in $m$ we observe that, if $L$ is a Lipschitz constant for  $K(x,\cdot)$, then $y\mapsto K(x,y)/L$ has Lipschitz constant 1, so by the very definition of Kantorovich-Rubinstein distance
\begin{equation*}\label{continuity}
\left| \int_{\overline{\Omega}} K(x,y) d(m(y) - \mu(y)) \right| = 
L \left| \int_{\overline{\Omega}} \frac{K(x,y)}{L}
d(m(y) - \mu(y)) \right| \leq L\dkr(m,\mu) .
\end{equation*}
\end{proof}
\begin{exe}[The basic game]
\label{exabas}
We consider the game with $N+N$ players and cost functions 
\begin{eqnarray*}
F^{1,N}_i(x_1, \ldots, x_N, y_1, \ldots, y_N) = V^{1,N}\left[\frac{1}{N-1}\sum_{i \neq j} \delta_{x_j}, \frac{1}{N}\sum \delta_{y_j} \right](x_i), \\
F^{2,N}_i(x_1, \ldots, x_N, y_1, \ldots, y_N) = V^{2,N}\left[\frac{1}{N}\sum_{j} \delta_{x_j}, \frac{1}{N-1}\sum_{j \neq i} \delta_{y_j} \right](y_i),
\end{eqnarray*}
where  $V^{k,N}$ are the regularized functionals \eqref{V1reg} and \eqref{V2reg} with $K\geq 0$ and Lipschitz, 
and $G$ 
is defined  by (\ref{eqG}).
Since $G(\gamma r,\gamma s; a, t) = G(r,s; a, \gamma^{-1}t)$ for all $\gamma \neq 0$,
\[
V^{k,N}\left[m_1, m_2 \right](x) = G\left(\int_{\overline{\Omega}} K(x,y) dm_k(y), \frac{N}{N-1} \int_{\overline{\Omega}} K(x,y) dm_{-k}(y); a_k, \eta_k \right) .
\]
Moreover, for $\eta_i>0$, $G$ is 
Lipschitz continuous in the first two entries, 
so we can pass to the limit as $N\to \infty$ 
and get  \eqref{FV} and  \eqref{FV2} with 
\begin{equation}
\label{Vk}
V^{k}\left[m_1, m_2 \right](x) := G\left(\int_{\overline{\Omega}} K(x,y) dm_k(y), \int_{\overline{\Omega}} K(x,y) dm_{-k}(y); a_k, \eta_k \right),
\end{equation}
where $m_{-1}=m_2$ and $m_{-2}=m_1$.
Furthermore, $(x,m_1,m_2)\mapsto V^k[m_1,m_2](x)$ are Lipschitz continuous by Lemma \ref{liplip}. Then Proposition \ref{nashlimits} applies to this example if there are Nash equilibria in pure strategies for the $N+N$ game, and in general Remark \ref{mixed} applies.
\end{exe}
\begin{exe}[Games with family effects]
\label{exafamily}
Here we take the cost functionals with ``family effects'' of Section \ref{family} and we regularize them as in Section \ref{more}, i.e.,  $V^{k,N}(x)$ are the regularized functionals \eqref{V1reg} and \eqref{V2reg}  as in the preceding example and we consider
\begin{equation}
\overline{F}^{k,N}_i(x_1, \ldots, x_N, y_1, \ldots, y_N) = \frac{1}{N} \sum_{l=1}^N V^{k,N}(x_l) W(x_i, x_l) ,
\end{equation}
where $W : \Rset^d\times\Rset^d\to \Rset$ is Lipschitz continuous. 
In this case, \eqref{FV} and  \eqref{FV2}  are satisfied by
\begin{equation}
\label{Vbarsmoothened}
\overline{V}^k[{m}_1, {m}_2](x) := \int_{\overline{\Omega}} W(x, z) V^{k}[{m}_1, {m}_2](z) d {m}_k (z)
\end{equation}
and $(x,m_1,m_2)\mapsto \overline{V}^k[m_1,m_2](x)$ are Lipschitz continuous as in the previous example.

Note that the functionals $V^k$ and $\overline{V}^k$ have a remarkably different behavior in areas where both populations are rare. In fact, assume that at some point $\bar x$ both $\int_{\overline{\Omega}} K(\bar x,y) dm_k(y)=0$ and, e.g, $\int_{\overline{\Omega}} W(\bar x,z) dm_1(z)=0$. Then
\[
{V}^1[{m}_1, {m}_2](\bar x)=a_1=\max G , \qquad \overline{V}^1[{m}_1, {m}_2](\bar x)=0 = \min G .
\]
\end{exe}
%
%
\subsection{Some explicit Mean-Field equilibria}
In this section, we give two simple examples of pairs $(\bar{m}_1, \bar{m}_2)\in  \Prob(\overline{\Omega})\times \Prob(\overline{\Omega})$ that satisfy the Mean-Field equations \eqref{MFoneshot2} (or, equivalently,  \eqref{MFoneshot}) for the basic game of Example \ref{exabas}. 
\begin{exe}[Uniform distributions]
In addition to the assumptions of Example \ref{exabas}, suppose that
\begin{equation}
\label{Kconst}
 \int_{\overline{\Omega}} K(x,y)\, dy = c \quad \text{ does not depend on } x.
\end{equation}
This says that the kernel $K$ gives the same total weight to the neighborhood $\U(x):=\text{supp} K(x,\cdot)$ of $x$, for all $x\in \overline{\Omega}.$
Consider the uniform distributions
\[
\bar{m}_1(x) = \bar{m}_2 (x) = 
1/{|\Omega|} \qquad \forall \, x\in  \overline{\Omega} ,
\]
where $|\Omega|$ denotes the measure of $\Omega$. Observe that, by \eqref{Kconst}, $V^k[\bar{m}_1, \bar{m}_2](x)$ is constant.
 Then the pair $(\bar{m}_1, \bar{m}_2)$ solves  \eqref{MFoneshot2} and therefore it is a Mean-Field equilibrium. Note that this occurs for all values of the parameters $
 a_k, \eta_k$, and that the ``value of the game'' $V^k[\bar{m}_1, \bar{m}_2](x)$ is not necessarily 0 (e.g., for $a_k\geq 1/2, \eta_k>0$).
\end{exe}
\begin{exe}[Fully segregated solutions]
 In addition to the assumptions of Example \ref{exabas}, we suppose now that, for some $r>0$,
 \begin{equation}
\label{suppKt}
\text{supp} K(x,\cdot) \subseteq \{z : |z-x|\leq r\}
\end{equation}
and $a_1, a_2<1$. We consider two sets $\Omega_1, \Omega_2 \subseteq \Omega$ such that 
\[
\text{dist}(\overline{\Omega}_1, \overline{\Omega}_2)\geq r , \qquad  \int_{\overline{\Omega}_k} K(x,y)\, dy\geq c_k>0 \quad\forall \,x\in \overline{\Omega}_k, \;k=1,2.
\]
 The second condition means that $\overline{\Omega}_k$ has enough weight near $x$ for all $x\in \overline{\Omega}_k$. We consider the distributions
 \begin{equation}
\label{segreg}
\bar{m}_1(x) = \left\{
\begin{array}{ll}
1/{|\Omega_1|} \quad & \quad\text{if } x\in\Omega_1 ,\\
0 & \qquad\text{else,}
\end{array}
\right.
\qquad \bar{m}_2(x) = \left\{
\begin{array}{ll}
1/{|\Omega_2|} \quad & \quad\text{if } x\in\Omega_2 ,\\
0 & \qquad\text{else.}
\end{array}
\right.
\end{equation}
In order to check 
  \eqref{MFoneshot2}, we first pick $x\in \text{supp }  \bar{m}_1=\Omega_1$. By \eqref{suppKt} and the first property of $\Omega_k$ we have
  \[
 \frac{ \int_{\overline{\Omega}} K(x,y)\, d\bar{m}_1(y)}{\int_{\overline{\Omega}} K(x,y)\, d\bar{m}_1(y) + \int_{\overline{\Omega}} K(x,y)\, d\bar{m}_2(y)+\eta_1}= 1/\left(1 + \frac{\eta_1 |\Omega_1|}{\int_{\overline{\Omega}_1} K(x,y)\, dy} \right),
  \]
and the right-hand side is above or equal to the threshold $a_1$ if and only if 
\[
\eta_1 |\Omega_1|\leq \int_{\overline{\Omega}_1} K(x,y)\, dy \left(\frac 1 {a_1}-1\right),
\]
which is true for all $x\in \Omega_1$ if 
\[
\eta_1 |\Omega_1| \frac{a_1}{1-a_1}\leq c_1 .
\]
Then for such values of the parameters $V^1[\bar{m}_1, \bar{m}_2](x) = 0$, so the first equation  \eqref{MFoneshot2} is satisfied. Similarly, if 
$    \eta_2 |\Omega_2| 
{a_2}/({1-a_2})\leq c_2$,  
for 
$x\in \text{supp }  \bar{m}_2=\Omega_2$ we have $V^2[\bar{m}_1, \bar{m}_2](x) = 0$ and also the second equation  \eqref{MFoneshot2} is verified. 
 Therefore we have a large set of parameters for which any segregated solution of the form \eqref{segreg} is a Mean-Field equilibrium.
 \end{exe}
 %
 \subsection{
 Models with 
  myopic players} 
  \label{locallimits}
In connection with the differential Mean-Field games  of the next sections, it is interesting to consider models where the cost functionals $V^k[m_1,m_2](x)$ depend only on $(m_1(x),m_2(x))$. This makes sense only if the measures $m_k$ have a density, and it is a limit case that does not meet the regularity conditions of Section \ref{largepop}. 
 We derive such local versions of the cost functionals by letting the size of the neighborhoods $ \U(x) $ tend to 0. This corresponds to individuals 
 who compute their cost functional by looking only at a very short distance, that we call {\em myopic players}.

Suppose 
that the kernel $K$ in Section \ref{more} 
 takes the form
\[
K(x,y) =  \rho^{-d} \varphi\left(\frac{x-y}{\rho}\right) 
\]
where 
$\varphi$ is a mollifier (i.e., a smooth nonnegative function $\Rset^d\to \Rset$ with support the unit ball centered at 0 and $\int_{\Rset^d} \varphi(z) dz=1$).
If $m \in L^1(\Omega)$, $\lim_{\rho\to 0}\int K(x,y) dm(y) =
m(x)$ for a.e. $x$. 

Consider first the functionals $V^k$ associated to the basic game (in the large population limit) defined by \eqref{Vk} in Example \ref{exabas}. Then
\begin{multline*}
\lim_{\rho\to 0}\ V^k[m_1, m_2](x) = 
G(m_k(x), m_{-k}(x); a_k, \eta_k) \\ = \left( \frac{m_k(x)}{m_k(x)+m_{-k}(x)+\eta_k} - a_k \right)^- =: V^k_{\ell}[m_1, m_2](x) .
\end{multline*}

Next we consider the game with family effects of Example \ref{exafamily} and assume 
 the kernel $W$ in \eqref{Vbarsmoothened} is also of the form
\[
W(x,y) =
 r^{-d} \psi\left(\frac{x-y}{r}\right) 
 \]
where $\psi$ is a mollifier. In the functionals $\overline{V}^k$ 
defined by \eqref{Vbarsmoothened}, we let first $r \rightarrow 0$ and get
\[
\lim_{r\to 0} \overline{V}^k [m_1, m_2](x) = m_k(x) V^k[m_1, m_2](x) .
\]
This a partially local model that can be interesting in some cases, but we do not study it further in this paper. 
Finally, we let 
 $\rho \rightarrow 0$ and obtain the local version of  $\overline{V}^k$:
\[
\lim_{\rho\to 0} \lim_{r\to 0} \overline{V}^k [m_1, m_2](x) = m_k(x) V^k_{\ell}[m_1, m_2](x)=: \overline{V_{\ell}}^k[m_1, m_2](x) .
\]
%

\section{Mean-field differential game models of segregation}
\label{differential}


\subsection{
Long-time average cost functionals}

In the last section, we designed some one-shot mean field games 
inspired by the original ideas of the population model by T. Schelling. We obtained the averaged costs $V^k, \overline{V}^k$ 
 by taking the limits as $N \rightarrow \infty$ of Nash equilibria of one-shot games with $2N$ players, and then the local limits $V^k_{\ell}, \overline{V}^k_{\ell}$ by shrinking the neighborhoods to points. We shall now investigate 
 dynamic  mean field games with the same cost functionals in a differential context. 
 We consider the state of 
 a representative agent of the $k$-th population governed by 
  the controlled stochastic differential equation with reflection
\begin{equation}\label{sdeXk}
dX_s^k=\alpha_s^k ds +\sqrt{2\nu} \, dB^k_s - n(X^k_s) dl^k_s,
\end{equation}
where  $B^k_s$ is a standard $d$-dimensional Brownian motion defined on some probability space, $\alpha_s^k$ is a control process adapted to $B^k_s$,
$n(x)$ is the outward normal to the open set $\Omega$ at the point $x\in\partial\Omega$, and the local time $l^k_s = \int_0^s \chi_{\partial \Omega} (X^k_s) dl^k_s$ is a non-decreasing process adapted to $B^k_s$. The term $n(X^k_s) dl^k_s$ 
 in the stochastic differential equation prevents the state variable $X^k_s$ to escape from $\overline{\Omega}$ by reflecting it when it reaches the boundary.

The goal of a player  of the $k$-th population is 
minimizing the long-time average cost, also called {\em ergodic cost},
\begin{equation}
\label{longtimeJ}
J^k(X^k_0, \alpha^1, \alpha^2, m_1, m_2) = \liminf_{T \rightarrow +\infty} \frac{1}{T} \mathbb{E} \left[\int_0^T L(X^k_s,\alpha^k_s) + V^k[m_1, m_2](X^k_s) ds \right],
\end{equation}
where $m^k$ are the 
 distributions of the two populations and $L$ is a Lagrangian function (smooth and convex in its second entry) 
 which represents the cost paid by the player for using the control $\alpha^k_s$ at the position $X^k_s$. 

The equilibrium distributions $m_k$ satisfy, together with $\lambda_k \in \Rset$ and the functions  $u_k$,  the stationary MFG system  of two Hamilton-Jacobi-Bellman  and two Kolmogorov-Fokker-Planck equations
\begin{equation}\label{MFGstat}
\left\{
\begin{array}{ll}
- \nu \Delta u_k + H(x,Du_k) + \lambda_k = V^k[m_1, m_2](x) & \text{in $\Omega$, $k=1,2$}\\
- \nu \Delta m_k - \div(D_p H(x,Du_k) m_k) = 0, \\
\partial_n u_k = 0, \quad \nu \partial_n m_k + m_k D_p H^k(x,Du_k)) \cdot n = 0, & \text{on $\partial \Omega$,}
\end{array}
\right.
\end{equation}
where the Hamiltonian $H$ is the Legendre transform of $L$ with respect to the 2nd entry, $\lambda_k$ 
 is the (constant) value of the representative agent of the $k$-th population, and the solutions $u_k$ of the H-J-B 
 equations provide the optimal strategies in feedback form $-D_p H(\cdot,Du_k(\cdot))$. 
Here the costs $V^k$ might be replaced by $\overline{V}^k$ or by the local versions $V^k_{\ell}$ and $\overline{V}^k_\ell$ defined in the previous section. 
 The connection between systems like \eqref{MFGstat}
and 
stochastic differential games  with $N$ players having the same dynamics and individual costs, as $N\to\infty$, was discovered by Lasry and Lions  Ref. \cite{LasryLions} in the periodic setting for a single population, and extended to several populations and more general data in Ref.  \cite{Feleqi} and to Linear-Quadratic problems in Ref. \cite{BP}, see also Ref. \cite{HCM:06} for related results by different methods.

Existence for \eqref{MFGstat} can be proved by means of fixed-point arguments when the cost functionals are bounded.

\begin{thm}\label{statexistence}
Let $\Omega$ be a convex domain. Suppose that $H(x,p) = R|p|^\gamma - H_0(x)$, where $R > 0, \gamma > 1$, $H_0 \in C^2 (\overline{\Omega})$ and $\partial_n H_0 \ge 0$ on $\partial \Omega$. Then, there exists at least one solution $(u_k,\lambda_k,m_k) \in C^{1, \delta}(\overline{\Omega}) \times \Rset \times W^{1,p}(\Omega)$ to \eqref{MFGstat} with costs either $V^k$, or $\overline{V}^k$, or $V^k_{\ell}$, $k=1,2$.
\end{thm}

\begin{proof} See Ref. \cite{Cirant}, Theorem 6
. \end{proof}
The case of local costs $\overline{V}^k_\ell$ in dimension $d > 1$ does not fit into the existence theorem because $\overline{V}^k_\ell$ is unbounded and a-priori estimates on solutions might fail in general. For space dimension $d = 1$ see Ref. \cite{CirantTesi}, Proposition 4.6. We do not expect uniqueness of the solution to the system \eqref{MFGstat}. 

For non-local $V^k$, $\overline{V}^k$ solutions can be proved to be classical and existence holds under weaker assumptions (see Theorem 4 
 in Ref. \cite{Cirant}), provided
 the negative part $(\cdot)^-$ in $G$ is replaced by some smooth regularization. 
We are interested in 
 qualitative properties of $m_1, m_2$, but no methods in this direction are known so far for solutions of PDE systems like  \eqref{MFGstat}. For such a reason, a numerical analysis will be carried out in Section \ref{sec_numresults}.

\subsubsection{The deterministic case in one space dimension}
\label{deterministiccase}

In order to convince ourselves that segregation phenomena might occur also in our differential MFG models, we briefly analyze the deterministic case $\nu=0$ in space dimension $d=1$. Suppose that the state space is a closed interval $\overline{\Omega} = [a,b] \subset \Rset$ and that there is no Brownian motion perturbing the dynamics of the average players ($\nu = 0$). Suppose also that $H(x,p) = |p|^2/2$. Then, \eqref{MFGstat} simplifies to
\begin{equation}\label{deterministicsys}
\left\{
\begin{array}{ll}
\frac{(u'_k)^2}{2} + \lambda_k = V^k[m_1, m_2](x) & \text{in $\Omega$, $k=1,2$} \\
(u'_k m_k)' = 0,\\
u'_k =0 ,\quad u'_k m_k = 0  & \text{on $\partial\Omega$,}
\end{array}
\right.
\end{equation}
where the Neumann boundary conditions must 
be interpreted in the viscosity sense, as it is natural when taking the limit as $\nu \rightarrow 0$.

It is possible to construct explicit solutions for this system. For simplicity, we will consider the non-smoothened costs 
\begin{equation*}
V^k[m_1, m_2](x) 
 = G\left(\int_{\U(x)} m_k, \int_{\U(x)} m_{-k}; a_k, \eta_k\right), \\
\end{equation*}
where $G$ is defined in \eqref{eqG} 
 and $m_{-1}=m_2, m_{-2}=m_1$.

\begin{exe}[Uniform distributions]
\[
m_k=\frac{1}{b-a}, \quad u_k = 0, \quad \lambda_k = V^k[m_1, m_2], \quad k=1,2
\]
provides a solution: the two populations are distributed uniformly and the cost functions are everywhere zero if the two thresholds $a_k$ are not large (say, below $.5$ if $\eta$ is negligible).
\end{exe}

\begin{exe}[Segregated solutions]
A family of 
fully segregated solutions may be written down explicitly. Suppose that $\U(x) = (x-r,x+r) \cap [a,b]$ with $r>0$ small, and let $a=x_0 < x_1 < x_2 < x_3 < x_4 < x_5 = b$ such that $x_{k+1}-x_k > r$ for $k=0,\ldots,4$. Set
\[
m_1(x) = \frac{1}{x_2-x_1} \chi_{[x_1, x_2]}(x), \quad m_2(x) = \frac{1}{x_4-x_3} \chi_{[x_3, x_4]}(x) \quad \forall x \in [a,b].
\]
Then, $\int_{\U(x)} m_1$ and $\int_{\U(x)} m_2$ are continuous functions which have support in $(x_1-r,x_2+r)$ and $(x_3-r,x_4+r)$, respectively. $V^1[m_1, m_2](\cdot)$ is also continuous, and vanishes in $[x_1, x_2]$ (if $a_1 < 1$ and $\eta_1$ is small enough); indeed, $\int_{\U(x)} m_2=0$, so $\int_{\U(x)} m_1 / \int_{\U(x)} (m_1 + m_2) = 1$. The same is for $V^2$, so we define
\[
\lambda_k = 0, \quad u_k(x) = \int_a^x (2V^k[m_1, m_2](\sigma))^{1/2} d\sigma, \quad \forall x \in [a,b], k=1,2.
\]
It is easy to see that the functions $(u_1,u_2)$ 
 verify the two HJB equations of \eqref{deterministicsys}. Moreover, they satisfy the Neumann boundary conditions $u_k'(a) = u_k'(b) = 0$ in the viscosity sense\footnote{A function $u\in C([a,b])$ satisfies the homogeneous Neumann boundary conditions in the viscosity sense in $a$ if, for all test functions $\phi \in C^2$ such that $u-\phi$ has a local maximum at $a$, then $\min\{(\phi'(a))^2 - 2V^1[m_1, m_2](a), \phi'(a)\} \le 0$, and  for all $\phi \in C^2$ such that $u-\phi$ has a local minimum at $a$, then $\max\{(\phi'(a))^2 - 2V^1[m_1, m_2](a), \phi'(a)\} \ge 0$.} (but not in classical sense, as $(u_k')^2= 2V^k \neq 0$ on the boundary of $[a,b]$); indeed, suppose that $\phi$ is a test function such that $u_1-\phi$ has a local maximum at $x=b$. If we set $s=(2V^1[m_1, m_2](b))^{1/2}$ it follows that $\phi'(b) \le s$. If $\phi'(b) \ge -s$ then $(\phi'(b))^2 \le s^2$, so
\[
\min\{(\phi'(b))^2 - 2V^1[m_1, m_2](b), \phi'(b)\} \le 0.
\]
Similarly, if $u_1-\phi$ has a local minimum at $x=b$,
\[
\max\{(\phi'(b))^2 - 2V^1[m_1, m_2](b), \phi'(b)\} \ge 0,
\]
and in the same way it also holds that $u'_1(a)=u'_2(a)=u'_2(b)=0$ in the viscosity sense.

It remains to check that $m_k$ are (weak) solutions of the two Kolmogorov equations. To do so, we notice that $m_1$ is zero outside $[x_1, x_2]$; in $[x_1, x_2]$, however, $V^1[m_1, m_2](x) = 0$, hence $u'_1(x)=0$. Similarly, $m_2(x)$ or $u'_2(x)$ vanishes, so $(u'_k m_k)'=0$.

\end{exe}

\subsection{Finite horizon problems}

When the the cost paid by a single player has the form \eqref{longtimeJ}, which captures the effect of the $m_k$ long-time average, the mean field system of partial differential equations \eqref{MFGstat} which characterizes Nash equilibria is stationary, i.e. no time dependance appears. Suppose, on the other hand, that a time horizon $T > 0$ is fixed, and the cost paid by the average player of the $k$-th population is of the form
\begin{equation}\label{fixedtimeJ}
J^k(X^k_0, t, \alpha^1, \alpha^2, m_1, m_2) = \mathbb{E} \left[\int_t^T L(x,\alpha^k_s) + V^k[m_1, m_2](X^k_s) ds  + G_T^k[m(T)](X^k_T)\right],
\end{equation}
where $t$ is the initial time and $G_T^k[m(T)]$ represents the cost paid at the final time $
T$. Then, the time variable $t$ enters the Mean Field Game system, which becomes

\begin{equation}\label{MFGnons}
\left\{
\begin{array}{ll}
- \partial_t u_k - \nu \Delta u_k + H^k(x,Du_k)  = V^k[m](x), & \textit{in } \Omega \times (0,T), \\
\partial_t m_k - \nu \Delta m_k - \div(D_p H^k(x,Du_k)m_k) = 0 & \textit{in } \Omega \times (0,T), \\
\partial_n u_k = 0, \, \nu \partial_n m_k + m_k D_p H^k(x, Du_k) \cdot n = 0  & \textit{on }  \partial \Omega \times (0,T), \\
u_k(x,T) = G_T^k[m(T)](x), \, m_k(x,0) = m_{k,0}(x) & \textit{in } \Omega
\end{array}
\right.
\end{equation}
We observe that \eqref{MFGnons} has a backward-forward structure: the Hamilton-Jacobi-Bellman equation for the value functions $V^k$ is backward in time, being the 
representative agent able to foresee the outcome of his actions, while his own distribution $m_k$ evolves forward in time. The final cost $G_T^k$ and the initial distributions $m_{k,0}$ are prescribed as final/initial boundary data.

For one population with periodic boundary conditions, the rigorous derivation of such a system from Nash equilibria of $2N$-persons games in the  limits as $N\to \infty$ was proved very recently in the fundamental paper by Cardaliaguet, Delarue, Lasry, and Lions  Ref. \cite{CDLL} on the so-called Master Equation of MFG. 
For related results by probabilistic methods, see Ref. \cite{Fish} and the references therein. The fact that from a solution of \eqref{MFGnons} one can synthesize $\epsilon$-Nash equilibria for the $2N$-persons game, if $N$ is large enough, is due to Huang, Caines and Malham{\'e} Ref. \cite{HCM:06}  (for one population) and to Nourian and Caines for problems with major an minor agents, Ref. \cite{NourCa}.

We also point out that the system \eqref{MFGstat} captures in some circumstances the behavior of \eqref{MFGnons} as $T \rightarrow \infty$. In particular, for a single population, 
if the cost $V$ is monotone increasing with respect to $m$, then solutions of \eqref{MFGnons} converge to solutions of \eqref{MFGstat} (see Ref. \cite{CardaliaguetLions}). 
It is not clear whether a similar phenomenon can be rigorously proved 
 in our multi-population systems, since monotonicity fails, but we show in Section \ref{sec_numresults} that it is likely to occur by providing some numerical evidences.

Existence of classical solutions for non-stationary Mean Field Games systems like \eqref{MFGnons} can be stated under rather general assumptions. In Ref. \cite{CardaliaguetNotes} a detailed proof is provided for the single-population case with periodic boundary conditions.
Next we state a precise existence result for our system \eqref{MFGnons} and outline its proof, whose main modifications are  due to the presence of Neumann boundary conditions.
Nevertheless, the general lines of the argument are the same: 
the fixed point structure of the system is exploited and the regularizing assumptions on $V^k, G_T^k$ assure that suitable a-priori estimates hold.

We recall that the space of probability measures $\Prob(\overline{\Omega})$ can be endowed with the Kantorovitch-Rubinstein distance, which metricize the weak$^*$ topology on $\Prob(\overline{\Omega})$. The assumptions on $V^k$, $G_T^k$, $m_{i,0}$ we require are
\begin{enumerate}
\item $V^k, G_T^k$ are continuous in $\overline{\Omega} \times \Prob(\overline{\Omega})^2$.
\item $V^k[m], G_T^k[m]$ are bounded respectively in $C^{1,\beta}(\overline{\Omega}), C^{2,\beta}(\overline{\Omega})$ for some $\beta > 1$, uniformly with respect to $m \in \Prob(\overline{\Omega})^2$.
\item $H^k \in C^1(\overline{\Omega} \times \RsetN)$ and it satisfies for some $C_0 > 0$ the growth condition
\[
D_pH^k(x,p) \cdot p \ge - C_0(1+|p|^2).
\]
\item $m_{i,0} \in C^{2,\beta}(\overline{\Omega})$.
\item The following compatibility conditions are satisfied:
\begin{align*}
&\partial_n G_T^k[m(T)](x) = 0, \quad \forall m \in  \Prob(\overline{\Omega})^2, x \in \partial \Omega, \\
&\partial_n m_{i,0}(x) + m_{i,0} D_p H^k(x, Du_k(x)) \cdot n = 0  \quad \textit{on }  \partial \Omega.
\end{align*}
\end{enumerate}
The assumptions 
(1) and (2) are satisfied by the non-local costs $V^k, \overline{V^k}$ defined by \eqref{Vk} and \eqref{Vbarsmoothened} in Section \ref{Exa}
if the negative part function  $(\cdot)^-$ in $G$ 
 is replaced by a smooth approximation \footnote{For example, $\varphi_\epsilon(t) = \frac{1}{2}( \sqrt{t^2 + \epsilon^2}-t)$,  $\epsilon > 0$ small, or $\Psi_{-, \epsilon}(\cdot)$ as in \eqref{approx_pnpart}.}.

\begin{thm}\label{nonstat_ex}
Under the  assumptions listed above there exists at least one classical solution to \eqref{MFGnons}.
\end{thm}

\begin{proof} \textbf{Step 1.} We start by an estimate on the Fokker-Planck equation. Suppose that $b$ is a given vector field, continuous in time and H{\"o}lder continuous in space (on $\overline{\Omega}$), and $m \in L^1(\Omega \times (0,T))$ solves in the weak sense 
\begin{equation}\label{genericFP}
\left\{
\begin{array}{ll}
\partial_t m - \nu \Delta m + \div(b \, m) = 0 & \textit{in } \Omega \times (0,T), \\
\nu \partial_n m(x) - m b \cdot n = 0  & \textit{on }  \partial \Omega \times (0,T), \\
m(x,0) = m_{0}(x) & \textit{in } \Omega.
\end{array}
\right.
\end{equation}

Then, $m(t)$ is the law of the following stochastic differential equation with reflection
\begin{equation}\label{reflectedsde}
\begin{array}{l}
X_t = X_0 + \int_0^t b(X_s, s) ds + \sqrt{2\nu} B_t - \int_0^t n(X_s) dl_s \quad X_t \in \overline{\Omega} \\
l_t = \int_0^t \chi_{\partial \Omega} (X_s) dl_s \\
l(0) = 0 \quad \text{$l$ is nondecreasing},
\end{array}
\end{equation}
where $B_t$ is a standard Brownian motion over some probability space, $X_t$, $l_t$ (the so-called local time) are continuous processes adapted to $B_t$ and the law of $X_0$ is $m_0$. This can be verified by exploiting the results of Ref. \cite{StrookVaradhanBoundary}, where it is proved that for all $\varphi \in C^2(\overline{\Omega})$ such that $\partial_n \varphi = 0$ on $\partial \Omega$,
\begin{equation}\label{marting}
M_t := \varphi(X_t) - \int_0^t [\nu \Delta \varphi(X_t)  + b(X_t, t) \cdot D\varphi(X_t)] dt
\end{equation}
is a martingale with respect to $B_t$. As a consequence, taking expectations in \eqref{marting} shows that the law of $X_t$ is the (unique) solution of \eqref{genericFP}.

This kind of stochastic interpretation of \eqref{genericFP} allows us to derive the following estimate:
\begin{multline*}
d(m(t), m(s)) = \sup \left\{ \int_{{\Omega}} \phi(x) (m(x,t) - m(x,s)) dx : \text{$\phi$ is $1$-Lipschitz continuous} \right\} \\
\le \sup \left\{ \E|\phi(X_t) - \phi(X_s)| : \text{$\phi$ is $1$-Lipschitz continuous} \right\} \le \E|X_t - X_s| \\
\le \E \left[ \int_s^t |b(X_\tau, \tau) d\tau| + \sqrt{2\nu} |B_t- B_s| \right],
\end{multline*}
for all $s,t \in [0,T]$, where the last inequality follows from Ref. \cite{AndersonOrey}. We can then conclude that
\begin{equation}\label{timeest}
d(m(t), m(s)) \le c_0 (1 + \|b\|_\infty)|t-s|^{\frac{1}{2}}
\end{equation}
for some $c_0$ which does not depend on $t,s$.

\textbf{Step 2.} We set up now the existence argument, which is based on a fixed-point method. Let $\C$ be the set of maps $\mu \in C^0([0,T],\Prob(\overline{\Omega}))$ such that
\begin{equation}\label{supnormC}
\sup_{s \neq t} \frac{d(\mu(s), \mu(t))}{|t-s|^{1/2}} \le C_1,
\end{equation}
for a constant $C_1$ large enough that will be chosen subsequently. The set $\C$ is convex and compact. To any $(\mu_1, \mu_2) \in \C^2$ we associate the (unique) classical solution $(u_1, u_2)$ of
\begin{equation}\label{hjbmap}
- \partial_t u_k - \nu \Delta u_k + H^k(x,Du_k)  = V^k[\mu_1, \mu_2](x),
\end{equation}
satisfying the Neumann boundary conditions $\partial_n u_k = 0$ on $\partial \Omega$, and then define $m = (m_1, m_2) = \Psi(\mu)$ as the solutions of the two Fokker-Planck equations
\begin{equation}\label{fpmap}
\partial_t m_k - \nu \Delta m_k - \div(D_p H^k(x,Du_k)m_k) = 0.
\end{equation}
A fixed point of $\Psi$ is clearly a solution of \eqref{MFGnons}. Such a mapping is indeed well-defined: existence for the HJB equation \eqref{hjbmap} is guaranteed by Theorem 7.4, p. 491 of Ref. \cite{Ladyzhenskaya} and the well-posedness of \eqref{fpmap} is stated in 
Theorem 5.3, p. 320 of Ref. \cite{Ladyzhenskaya}. These results incorporate also the Schauder a-priori estimates, that together with \eqref{timeest} make $\Psi$ continuous and a mapping from $\C^2$ into itself, provided that the constant $C_1$ in \eqref{supnormC} is large enough. The existence of a fixed point for $\Psi$ follows from the application of the Schauder fixed point theorem.
\end{proof}

While existence of smooth solutions of \eqref{MFGnons} with costs $V^k, \overline{V^k}$ can be established through standard methods, the local versions $V^k_\ell, \overline{V}_\ell^k$ are not regularizing, so the ideas of Theorem \ref{nonstat_ex} cannot be applied directly; in this case, existence of solutions is a much more delicate issue. 

A well-established workaround is to smoothen the costs by convolution with kernels, and pass to the limit in a sequence of approximating solutions (which are obtained by arguing as in Theorem \ref{nonstat_ex}); this procedure requires a-priori bounds, that strongly depend on the behavior of the Hamiltonian  at infinity, the cost, and the space dimension $d$. It is not the purpose of this paper to present theoretical results on existence of smooth solutions in full generality. 
We believe that, under suitable assumptions, solutions can be obtained without substantial difficulties 
 by extending 
 known results for one-population MFG on the torus to the case of two populations with Neumann boundary conditions. Next we briefly explain how. 

Suppose that $H^k(x,p)$ behaves like $c |p|^\gamma$ as $p \to \infty$ ($c > 0$, and $\gamma > 1$). In our setting, the couplings $V^k_\ell, \overline{V}_\ell^k$ are non-negative, and a-priori bounds on $\int |D u_k|^\gamma m_k \, dx dt$ and $\int V^k_\ell m_k \, dx dt$ (quantities that are somehow related to the energy of the system) can be easily proved. To carry out the approximation procedure, it is crucial to have a-priori bounds on $\|m_k\|_{L^\infty(\Omega)}$.

\begin{exe} 
In the purely quadratic case, namely, 
 $H^k(x,p) = |p|^2/2$, the Hopf-Cole transformation can be used to  transform \eqref{MFGnons} into a system of two couples of semilinear equations of the form
\[
\begin{cases}
& - \partial_t \phi_k - \nu \Delta \phi_k + \frac{1}{2 \nu} V^k_\ell(\phi_1 \psi_1, \phi_2 \psi_2) \phi_k = 0, \\
& \partial_t \psi_k - \nu \Delta \psi_k + \frac{1}{2 \nu} V^k_\ell(\phi_1 \psi_1, \phi_2 \psi_2) \psi_k = 0,
\end{cases}
\]
where $\phi_k = e^{-u_k/{2\nu}}$ and $\psi_k = m_k e^{u_k/{2\nu}}$, 
 with the corresponding initial-final data and Neumann boundary conditions. Bounds on $\|m_k\|_{L^\infty(\Omega)} = \|\phi_k \psi_k\|_{L^\infty(\Omega)}$ can be derived by arguing as in Ref. \cite{CardaliaguetLions}, where a Moser iteration method is implemented.
\end{exe}

\begin{exe} If $1 < \gamma < 1+1/(d+1)$, 
so that $H$ grows almost linearly, it is known that existence of smooth solutions can be established, see the discussion in Ref. \cite{GPMsub}. In particular, the basic estimate for $\int |D u_k|^\gamma m_k$ implies that the drifts $D_p H^k$ entering the Fokker-Planck equations belong to $L^p(m_k)$, where $p > d+2$. It is known that this kind of Lebesgue regularity on the drifts is strong enough to guarantee H{\"o}lder bounds for $m_k$.
\end{exe}

\begin{exe} 
For other values of $\gamma$, we observe that  $V^k_\ell$ 
are uniformly bounded. 
 Therefore,  at least in the subquadratic case (namely, when $\gamma \le 2$), one might exploit the classical Lipschitz bounds for viscous HJ equations and H\"older estimates for the Fokker-Planck to achieve a-priori regularity for $m_k$, see Ref. \cite{Ladyzhenskaya}. 
 
The setting with the costs $\overline{V}_\ell^k$ is 
 more delicate, as $\overline{V}_\ell^k$ is a-priori unbounded in $L^\infty$. Here, one might reason as in Ref. \cite{GPMsub}, or Ref. \cite{GPMsup} in the superquadratic case (see also Ref. \cite{GomesBook}), and finely combine regularity of the HJB equation and the Fokker-Planck equation to prove existence of solutions of \eqref{MFGnons}, at least if the space dimension is sufficiently small ($d = 1,2$). We leave these extensions to future work. 
\end{exe} 

\section{Numerical methods}
\label{sec_numstrategies}


Numerical methods for approximating mean field game systems are an important research issue
 since they are crucial for applications.
The finite difference methods described below are reminiscent ot the method first introduced and analysed in Ref. \cite{AchdouCapuzzo}
for  mean field games with a single population, which, to the best of our knowledge, 
remains the more robust and flexible technique.   
The numerical scheme  basically relies on monotone approximations of the Hamiltonian and on a suitable weak formulation of the Kolmogorov equation.
 It has several important features:  
\\
\begin{itemize}
\item  existence and possibly uniqueness for the discretized problems can be obtained by similar arguments as those used in the continuous case
\item  it is robust when $\nu\to 0$ (the deterministic limit of the models)
\item it can be used  for finite and infinite horizon problems
\item bounds on the solutions, which are uniform in the grid step, can be proved under reasonable assumptions on the data.
\end{itemize}
A first result on  the convergence to classical solutions was contained in Ref. \cite{AchdouCapuzzo}. The method was used for planning problems 
(the terminal condition is a Dirichlet like condition for $m$) in  Ref. \cite{AchdouPlanning}.
 Ref. \cite{AchdouCamilli} contains a further analysis of convergence to classical solutions and very general  results on the convergence to weak solutions are supplied
 in  Ref.  \cite{AP2016}. In Ref. \cite{ABLLM14}, similar computational techniques are  applied to MFG models in macro-economics.
\\
Discrete time, finite state space mean  field games
were discussed in  Ref. \cite{MR2601334}. We also refer to Ref. \cite{GueantQuadratic,MR2928382}
 for a specific constructive approach when the Hamiltonian is quadratic.  Semi-Lagrangian approximations were investigated
 in Ref. \cite{MR3148086,CS2014}. Finally, augmented Lagrangian methods for the solution of the  system of equations arising from the discrete version of a variational 
mean field game was proposed in Ref. \cite{BenamouCarlier2015ALG2}.

\subsection{Stationary PDEs}\label{s:stat_approx}


To approximate \eqref{MFGstat}, we will implement the strategy 
proposed in Ref. \cite{AchdouCapuzzo}, that consists of taking the long-time limit of the \textit{forward-forward} MFG system
\begin{equation}\label{ffsys}
\begin{cases}
\partial_t u_k - \nu \Delta u_k + H^k(x,Du_k) = V^k[m_1, m_2](x) & (0,T) \times \Omega\\
\partial_t m_k - \nu \Delta m_k - \div(D_p H^k(x, Du_k ) \, m_k) = 0, \\
\partial_n u_k = 0, \quad \nu \partial_n m_k + m_k D_p H^k(x, Du_k) \cdot n = 0, & (0,T) \times \partial \Omega \\
u_k(t=0) = u_{k,0}, \quad m_k(t=0) = m_{k,0}, \quad k=1,2.
\end{cases}
\end{equation}
This method is reminiscent of long-time approximations for the cell problem in homogenization theory: we expect that 
there exists some $\lambda_k \in \Rset$ such that $u_k(\cdot,T)- \lambda_k T$ and $m_k(\cdot,T)$ converge as $T \to \infty$, 
respectively, to some $\bar{u}_k(\cdot),\bar{m}_k(\cdot)$ solving \eqref{MFGstat}. Although this has not been proven rigorously in general in the MFG setting, 
Gu{\'e}ant studies some single-population examples where the coupling $V(m)$ is not increasing with respect to the distribution $m$
 (so there is no uniqueness of solutions, as in our framework) and justifies the approach (see Ref. \cite{Gueantnotes}). Very recently, a proof of the long-time convergence for a class of forward-forward one dimensional MFG has been proved in Ref. \cite{GomesFF}.
 We are going to present  numerical experiments, even if  no rigorous proof of any convergence is available at this stage in our multi-population setting.

We mention that if the Hamiltonians $H^k$ are quadratic, it is possible to simplify \eqref{MFGstat} through the Hopf-Cole change of variables and reduce the number of unknowns (see Ref. \cite{GueantQuadratic}).

We will develop a finite-difference scheme for \eqref{ffsys} in space dimension $d=2$ as in Ref. \cite{AchdouCapuzzo}, assuming for simplicity that the Hamiltonians are of the form
\begin{equation}\label{modelH}
H^k(x,p) = W^k(x) + \frac{1}{\gamma_k}|p|^{\gamma_k}, \quad \gamma_k > 1, \quad W^k \in C^2(\Omega).
\end{equation}

In space dimension $d \neq 2$, analogous schemes can be set up. Consider a square domain $\Omega = (0,1)^2$, and a uniform grid with mesh step $h$, assuming that $1/h$ is an integer $N_h$; denote by $x_{i,j}$ a generic point of the grid. Let $\Delta t$ be a positive time step   and  $t_n = n \Delta t$. 
The values of $u_k$ and $m_k$ at $x_{i,j}$, $t_n$ will be approximated by $U_{i,j}^{k,n}$ and $M_{i,j}^{k,n}$ respectively, $k=1,2$, $i,j = 1,\ldots,N_h$ 
and $n\ge 0$.
\\
We introduce the usual finite difference operators
\[
(D^+_1 U)_{i,j}=\frac{U_{i+1,j}-U_{i,j}}{h}, \quad (D^+_2 U)_{ij}=\frac{U_{i,j+1}-U_{i,j}}{h},
\]
and the numerical Hamiltonians $g^k: \Omega \times \Rset^4 \rightarrow \Rset$ of Godunov type defined by 
\[
g^k(x,q_1,q_2,q_3,q_4) = W^k(x) + \frac{1}{\gamma_k}\left[[(q_1)^-]^2 + [(q_3)^-]^2 + [(q_2)^+]^2 + [(q_4)^+]^2\right]^{\gamma_k/2}.
\]
Denoting by
\[
[D_h U]_{i,j} = ((D^+_1 U)_{i,j}, (D^+_1 U)_{i-1,j}, (D^+_2 U)_{i,j}, (D^+_2 U)_{i,j-1}),
\]
the finite difference approximation of the Hamiltonian function $H^k$ will be $g^k(x,[D_h U^k]_{i,j})$. \\
We choose the classical five-points discrete version of the Laplacian
\[
(\Delta_h U)_{i,j} = -\frac{1}{h^2}(4U_{i,j}-U_{i+1,j}-U_{i-1,j}-U_{i,j+1}-U_{i,j-1}).
\]
The non-local couplings $V^k[m_1, m_2]$, $\overline{V}^k[m_1, m_2]$ involve terms of the form $\int_\Omega K(x,y)m_k(y) dy$; we approximate them via
\[
h^2\sum_{r,s} K(x_{i,j}, x_{r,s}) M_{r,s}^{k,n}.
\]
On the other hand, local couplings $V_\ell^k$ and $\overline{V}_\ell^k$ will be simply function evaluations at $x_{i,j}$, that is $(V_\ell^k[M^{1,n}, M^{2,n}])_{i,j} = V^k_\ell(M^{1,n}_{i,j}, M^{2,n}_{i,j})$.

In order to approximate the Kolmogorov equations in \eqref{ffsys}, we consider their weak formulation. Given any test function $\phi$, the divergence term involved can be rewritten as
\[
-\int_\Omega \div (m_k D_p H^k(x,Du_k)) \phi = \int_\Omega m D_p H^k(x,Du_k) \cdot D \phi,
\]
which is going to be approximated by (boundary terms disappear by Neumann conditions)
\[
h^2 \sum_{i,j} M^{k,n}_{i,j} D_q g^k(x,[D_h U^{k,n}]_{i,j}) \cdot [D_h \Phi]_{i,j},
\]
where $\Phi$ is the finite difference version of $\phi$. By introducing the compact notation
\[
\BB^k_{i,j}(U,M) = \frac{1}{h} \left( \begin{array}{l}
M_{i,j} \partial_{q_1} g^k(x,[D_h U]_{i,j}) - M_{i-1,j} \partial_{q_1} g^k(x,[D_h U]_{i-1,j}) \\
\quad + M_{i+1,j} \partial_{q_2} g^k(x,[D_h U]_{i+1,j}) - M_{i,j} \partial_{q_2} g^k(x,[D_h U]_{i,j}) \\
\quad\quad + M_{i,j} \partial_{q_3} g^k(x,[D_h U]_{i,j}) - M_{i,j-1} \partial_{q_3} g^k(x,[D_h U]_{i,j-1}) \\
\quad\quad\quad + M_{i,j+1} \partial_{q_4} g^k(x,[D_h U]_{i,j+1}) - M_{i,j} \partial_{q_4} g^k(x,[D_h U]_{i,j})
\end{array}\right),
\]
we can finally write the discrete version of \eqref{ffsys}
\begin{equation}\label{ffdiscrsys}
\left\{
\begin{array}{l}
\frac{U_{i,j}^{k,n+1}-U_{i,j}^{k,n}}{\Delta t} - \nu (\Delta_h U^{k,n+1})_{i,j} + g^k(x,[D_h U^{k,n+1}]_{i,j}) = (V^k[M^{1,n+1}, M^{2,n+1}])_{i,j}, \\
\frac{M_{i,j}^{k,n+1}-M_{i,j}^{k,n}}{\Delta t} - \nu (\Delta_h M^{k,n+1})_{i,j} - \BB^k_{i,j} (U^{k,n+1}, M^{k,n+1}) = 0, \quad k=1,2.
\end{array}
\right.
\end{equation}
The system above has to be satisfied for internal points of the grid, i.e. $2 \le i,j \le N_h - 1$. The finite difference version of the homogeneous Neumann boundary conditions for $U$ is, for all $n,k$,
\[
\begin{split}
& U^{k,n}_{1,j} = U^{k,n}_{2,j}, \quad U^{k,n}_{N_h - 1,j} = U^{k,n}_{N_h,j}, \quad \forall j =2,\dots,N_h-1 \\
& U^{k,n}_{i,1} = U^{k,n}_{i,2}, \quad U^{k,n}_{i,N_h - 1} = U^{k,n}_{i,N_h}, \quad \forall i =2,\dots,N_h-1 \\
& U^{k,n}_{1,1} = U^{k,n}_{2,2}, \quad U^{k,n}_{N_h,1} = U^{k,n}_{N_h-1,2}, \\
& U^{k,n}_{1,N_h} = U^{k,n}_{2,N_h-1}, \quad U^{k,n}_{N_h,N_h} = U^{k,n}_{N_h-1,N_h-1}.
\end{split}
\]
In a similar manner, boundary conditions will be imposed on $M^{k,n}$ (note that, in view of the particular choice of the Hamiltonian, $\partial_n m_k = 0$ on the boundary); 
The scheme guarantees that $M^{k,n}_{i,j} \ge 0$.

In Ref. \cite{AchdouCapuzzo} it is proven  that \eqref{ffdiscrsys} has a solution in the case of a single population and 
 periodic boundary conditions, 
 (see Theorem 5). We expect that it is true also with Neumann boundary conditions and two populations, since similar arguments can be used.

The present scheme  is implicit, since each time iteration consists of solving a coupled system of nonlinear equations for $U^{k,n+1}, M^{k,n+1}$, given $U^{k,n}, M^{k,n}$. This can be done for example by means of a Newton method, increasing possibly the time step when the asymptotic regime is close to be reached. It has been indicated in Ref. \cite{AchdouCapuzzo}, Remark 11, that in order to have a good approximation of the system of nonlinear equations, it is sufficient to perform just one step of the Newton method: indeed, it has been observed that in general one step reduces the residual substantially.

Finally, the discrete version of \eqref{ffsys} that will be implemented for numerical experiments reads
\begin{equation}\label{ffdiscrsystwo}
\left\{
\begin{array}{l}
\frac{U_{i,j}^{k,n+1}-U_{i,j}^{k,n}}{\Delta t} -  \nu (\Delta_h U^{k,n+1})_{i,j} + g^k(x,[D_h U^{k,n}]_{i,j})  \\
 \hspace{3cm}+ D_q g(x,[D_h U^{k,n}])_{i,j} \cdot ([D_h U^{k,n+1}]_{i,j} - [D_h U^{k,n}]_{i,j}) \\
 \hspace{8cm} = (V^k[M^{1,n}, M^{2,n}])_{i,j}, \\
\frac{M_{i,j}^{k,n+1}-M_{i,j}^{k,n}}{\Delta t} -  \nu (\Delta_h M^{k,n+1})_{i,j} - \BB_{i,j}^k(U^{k,n+1}, M^{k,n+1}) = 0, \quad k=1, 2.
\end{array}
\right.
\end{equation}
In this formulation, at each time iteration one needs to solve a coupled system of linear equations. Note that \eqref{ffdiscrsystwo} consists of an implicit scheme for the (forward) Kolmogorov equation (i.e. implicit with respect to $m$ and $u$), 
coupled with a \emph{linearized} semi-implicit scheme for the (forward) Hamilton-Jacobi equation (i.e. implicit with respect to $u$ and explicit with respect to $m$). \\
We choose the initial data
\[
U^{k,0} = 0, \quad M^{k,0} = M^k_0,
\]
with
\[
h^2 \sum_{i,j} (M^k_0)_{i,j} = 1, \quad k=1,2.
\]
We expect that  there exists some real number $\lambda_{h,\Delta t}$,
such that  $M^{k,n}$ and   $U^{k,n} - \lambda_{h,\Delta t} n\Delta t $  tend to some stationary configuration as $n$ tends to infinity.

\subsection{Evolutive   PDEs} 
The discrete scheme used for (\ref{MFGnons}) is obtained by adapting the methods proposed and studied in Ref. \cite{AchdouCapuzzo}  
to the multi-population case.  For simplicity, 
let us focus on the case when the terminal cost for the agents of type $k$ does not
 depend on $m(T)$, so the terminal condition on $u_k$ becomes
\begin{displaymath}
  u_k(x,T)= u_{k,T}(x) \quad \hbox{in } \Omega,
\end{displaymath}
and on Hamiltonians given by (\ref{modelH}).
The time-step $\Delta t$ is assumed to be of the form $T/N$, for a positive integer $N$. 
Using the same notations as in \S~\ref{s:stat_approx}, the approximate version of 
\eqref{MFGnons}  reads:
 for any $0\le n < N$, $1< i,j< N_h$,
\begin{equation}\label{eq:1}
\left\{
  \begin{array}[c]{ll}
\ds \frac{U_{i,j}^{k,n+1}-U_{i,j}^{k,n}}{\Delta t} + \nu (\Delta_h U^{k,n})_{i,j} - g^k(x,[D_h U^{k,n}]_{i,j}) 
 &\ds =- (V^k[M^{1,n}, M^{2,n}])_{i,j},\\
\ds\frac{M_{i,j}^{k,n+1}-M_{i,j}^{k,n}}{\Delta t} -  \nu (\Delta_h M^{k,n+1})_{i,j} - \BB^k_{i,j}(U^{k,n}, M^{k,n+1}) &= 0, \quad k=1, 2,
  \end{array}
\right.
\end{equation}
with the initial and terminal conditions: for $1\le i,j\le N_h$,
\begin{equation}\label{eq:2}
M_{i,j}^{k,0}= m_{k,0}(x_{i,j}),\quad  \quad   U_{i,j}^{k,N}= u_{k,T}(x_{i,j}).
\end{equation}
It can be supplemented with discrete Neumann  conditions as in \S~\ref{s:stat_approx} or with periodicity conditions.
Note that (\ref{eq:1}) consists of a semi-implicit scheme for the  (forward) Kolmogorov equation 
 (i.e. implicit with respect to $m$ and explicit with respect to $u$) 
coupled with a semi-implicit scheme for the (backward)  Hamilton-Jacobi equation (i.e. implicit with respect to $u$ and explicit with respect to $m$). 
When dealing with one population only, it was shown in Ref. \cite{AchdouCapuzzo} that the discrete scheme
 preserves the structure of the continuous problem, which makes it possible to prove existence, and  uniqueness/stability 
under additional assumptions. In the multi-population case also, existence of solutions of the discrete system
 can be obtained by using a Brouwer fixed point method. Then, assuming that   $ h^2 \sum_{i,j}  M_{i,j}^{k,0}=1$ for $k=1,2$,
 mass conservation, i.e. $h^2 \sum_{i,j}  M_{i,j}^{k,n}=1$  for any $n$, $k=1,2$,
 is a consequence of the definition of $\BB^k$. Using the monotonicity of $g$, we also obtain the nonnegativity of $M^{k,n}$ for any $n$, $k=1,2$, see Ref. \cite{AchdouCapuzzo}. 
\\
We briefly describe the iterative method used in order to solve (\ref{eq:1})-(\ref{eq:2}).
Since the latter system couples  forward and backward (nonlinear) equations, 
it cannot be solved by merely marching in time. 
Assuming that the discrete Hamiltonians are $C^2$ and the coupling functions are $C^1$ 
 allows us to use a Newton-Raphson method for the whole 
system of nonlinear equations  (which can be huge if  $d\ge 2$). \\
  More precisely, we see (\ref{eq:1})-(\ref{eq:2}) as a fixed point problem.
  We first define the mapping $\Xi$ which maps the pair of grid functions 
$\left(Y^{1,n}_{i,j}, Y^{2,n}_{i,j}\right)_{i,j,n}$  
 to the pair of grid function $\left( (V^1[M^{1,n}, M^{2,n}])_{i,j}\right.$, $\left.(V^2[M^{1,n}, M^{2,n}])_{i,j}     \right)_{i,j,n}
$,   where $n$ takes its values in  $\{1\dots, N\}$ and   $i,j$ take their values  in  $\{1\dots, N_h\}$,
and $(M^{1,n}_{i,j}, M^{2,n}_{i,j})$ is found by solving
 the following system of discrete Bellman and Kolmogorov equations: for any $0\le n< N$, $1< i,j< N_h$,
\begin{equation}\label{eq:3}
\left\{
  \begin{array}[c]{ll}
\ds \frac{U_{i,j}^{k,n+1}-U_{i,j}^{k,n}}{\Delta t} +  \nu (\Delta_h U^{k,n})_{i,j} - g^k(x,[D_h U^{k,n}]_{i,j}) 
 &\ds =   - Y^{k,n+1}_{i,j} ,\\
\ds\frac{M_{i,j}^{k,n+1}-M_{i,j}^{k,n}}{\Delta t} -  \nu (\Delta_h M^{k,n+1})_{i,j} - \BB_{i,j}^k(U^{k,n}, M^{k,n+1}) &= 0,
  \end{array}
\right.
\end{equation}
 supplemented with (\ref{eq:2}) and discrete Neumann conditions. Finding a fixed point of $\Xi$ is equivalent  
to solving (\ref{eq:1})-(\ref{eq:2}).\\
Note that in (\ref{eq:3}) the discrete Bellman equations do not involve $M^{k,n+1}$.
Therefore, one can first solve the Bellman equations for $U^{k,n}$ $0\le n\le N$, $k=1,2$ 
by marching backward in time (i.e. performing a backward loop with respect to the index $n$).
For every time index $n$, the two systems of nonlinear equations for $U^{k,n}$, $k=1,2$ are themselves solved by means 
of a nested Newton-Raphson method. 
 Once an approximate solution of the Bellman  equations has been found, 
one can solve the (linear) Kolmogorov equations for $M^{k,n}$ $0\le n\le N$, $k=1,2$,  
by marching forward in time (i.e. performing a forward loop with respect to the index $n$).  
The solutions of (\ref{eq:3})-(\ref{eq:2}) are such that $M^{k,n}$ are nonnegative and 
$h^2 \sum_{i,j}  M_{i,j}^{k,n}=1$  for any $n$, $k=1,2$.
\\
 The fixed point equation 
$\Xi \left( \left(Y^{1,n}_{i,j}, Y^{2,n}_{i,j}\right)_{i,j,n}\right)=\left(Y^{1,n}_{i,j}, Y^{2,n}_{i,j}\right)_{i,j,n}$ is solved numerically 
by using a Newton-Raphson method. This  requires the differentiation of both the Bellman and Kolmogorov equations in (\ref{eq:3}).
\\
A good choice of an initial guess  is important, as always for  Newton methods. 
To address this matter, we first observe that the 
above mentioned iterative  method generally quickly converges to a solution when the value of $\nu$ is large.
This leads us to use a continuation method in the variable $\nu$:
 we start solving  (\ref{eq:1})-(\ref{eq:2}) 
with a rather high value of the parameter $\nu$ (of the order of $1$),
 then gradually decrease $\nu$ down to the desired value,
 the solution found for a  value of $\nu$ being used as an initial guess
 for the iterative solution with the next and smaller value of $\nu$.

\section{Numerical simulations}\label{sec_numresults}

\subsection{Stationary  PDEs}

In this section, we will show some results obtained by implementing the long-time procedure presented in Section \ref{s:stat_approx}. Here, we choose $d = 1$, $\Omega = (0,1)$ and Hamiltonians of the form \eqref{modelH}, with $W \equiv 0$. The mesh step is $h = 1/200$; at each time step $n$ we define the approximate ergodic constant $\lambda_k^n = h (\sum_{i} U^{k,n}_i)/t_n$ and the relative errors $err^n_m = \max_{k=1,2} \|M^{k,n} -  M^{k,n-1}\|_\infty/\Delta t$, $err^n_\lambda = \max_{k=1,2} |\lambda_k^n- \lambda_k^{n-1}|$. As mentioned before, we expect that as $t_n$ grows, $\lambda_k^n$ converges to some constant value; we stop the simulation when the two relative errors become smaller than a fixed threshold, and denote by $u^k_h, m^k_h$ the approximate solutions $U^{k,n}, M^{k,n}$ respectively at the last time iteration.

The initial data are set to be (unless otherwise specified)
\[
U^{k,0} \equiv 0, \quad M^{1,0}_i = \chi_{[0, 0.5]}(x_i), \quad M^{2,0}_i = \chi_{[0.5, 1]}(x_i),
\]
while the time step is $\Delta t = 0.02$ as long as the relative error is large, namely when $err_m > 1$ (this happens during the first time iterations), and it is linearly increased to $\Delta t = 2$ as soon as the relative error $err_m$ reaches $0.001$. In our simulations, stability in the long-time regime always occurs; in Figure \ref{fig_test1_2} (right) it is shown a typical behavior of the relative errors as the number of time iterations increases.

We will show various tests with different values of $H, \nu$, and different choices of the cost functionals (see Table \ref{tests}). Note that if $\nu$ is large (say, greater than $0.1$), the constant solution only is achieved in the long-time regime, namely $M^{k,n} \to 1$ as $n$ increases; in this situation the mixing effect of the Brownian noise prevails on the individual preference of players. A richer structure of approximate solutions shows up as $\nu$ approaches zero.

\begin{table}[ht]
\centering
\caption{The data in the tests.}\label{tests}
\begin{tabular}{|C|C|C|C|C|C|}
\hline
\text{Test} & \gamma & \nu & a_1 & a_2 & \text{Couplings} \\ \hline \hline
1 & 2 & 0.05, 0.0005 & 0.3 & 0.4 & V_\ell \\ \hline
2 & 2 & 0.05 & 0.4 & 0.8 &  \overline{V}_\ell, V_\ell \\ \hline
3 & 2 & 0.001 & 0.8, 0.3 & 0.8, 0.3  & V \\ \hline
4 & 8, \frac{4}{3} & 0.005 & 0.3 & 0.3 & V_\ell \\ \hline
\hline
\end{tabular}
\end{table}

\smallskip {\bf Test 1.} Here, we obtain two monotone configurations, and observe that segregation between the two populations appears; moreover, it becomes more evident as the viscosity $\nu$ goes to zero, see Figure \ref{fig_test1}. In other words, we find two disjoint intervals $\Omega_k$, $k=1,2$ such that $m_h^k > 0$ on $\Omega_k$ and $m_h^{3-k} \to 0$ as $\nu \to 0$ on $\Omega_k$. Note that segregation occurs even if the two ``happiness'' thresholds $a_k$ are small: the cost paid by a player can be zero even if the distribution of his own population is less than half of the distribution of both the populations. The optimal feedback control $- D_h u_h^k$ vanishes on $\Omega_k$ in the small viscosity regime, because in this region the cost $V^k_\ell$ is identically zero; $- D_h u_h^k$ acts substantially only on the complement of $\Omega_k$, forcing $m_h^k$ to be close to zero.

Note that if $\nu$ is very small, the free boundary between $\Omega_1$ and $\Omega_2$ becomes a point, which varies upon the choice of $a_k$ (see also the other tests); in general, if $a_1 > a_2$ this boundary shifts closer to $x=0$ if $0 \in \Omega_1$, or to $x=1$ if $1 \in \Omega_1$: the more xenophobic population concentrates more, while the other one is distributed over a bigger subset of the domain.

The asymptotic behavior of $\int m^1 m^2 \, dx$ with respect to $\nu$ appears to be power-like, that is $\int m^1 m^2 \, dx \approx c \nu^4$ for some positive $c$, depending on the ``branch'' of solutions. For this test, numerical values can be found in Table \ref{nu_vs_int}.

We finally mention that if one changes the initial distributions $M^{1,0}, M^{2,0}$, then the approximate solution $m^k_h$ may vary; in the one dimensional case monotone configurations are likely to occur, but it is possible to obtain solutions with more than one stationary point (see Figure \ref{fig_test1_2}) by a suitable choice of $M^{k,0}$ (see also Remark \ref{rem_CV}).

\begin{figure}
\centering
\caption{\footnotesize $m_h$ (left), $u_h$ (right) at different values of $\nu$: $\nu = 0.05$ is marked with circles, $\nu = 0.0005$ is marked with triangles; solid/red lines are used for $(u_1,m_1)$, while dashed/blue lines are used for  $(u_2,m_2)$.}\label{fig_test1}
    \includegraphics[width=6cm]{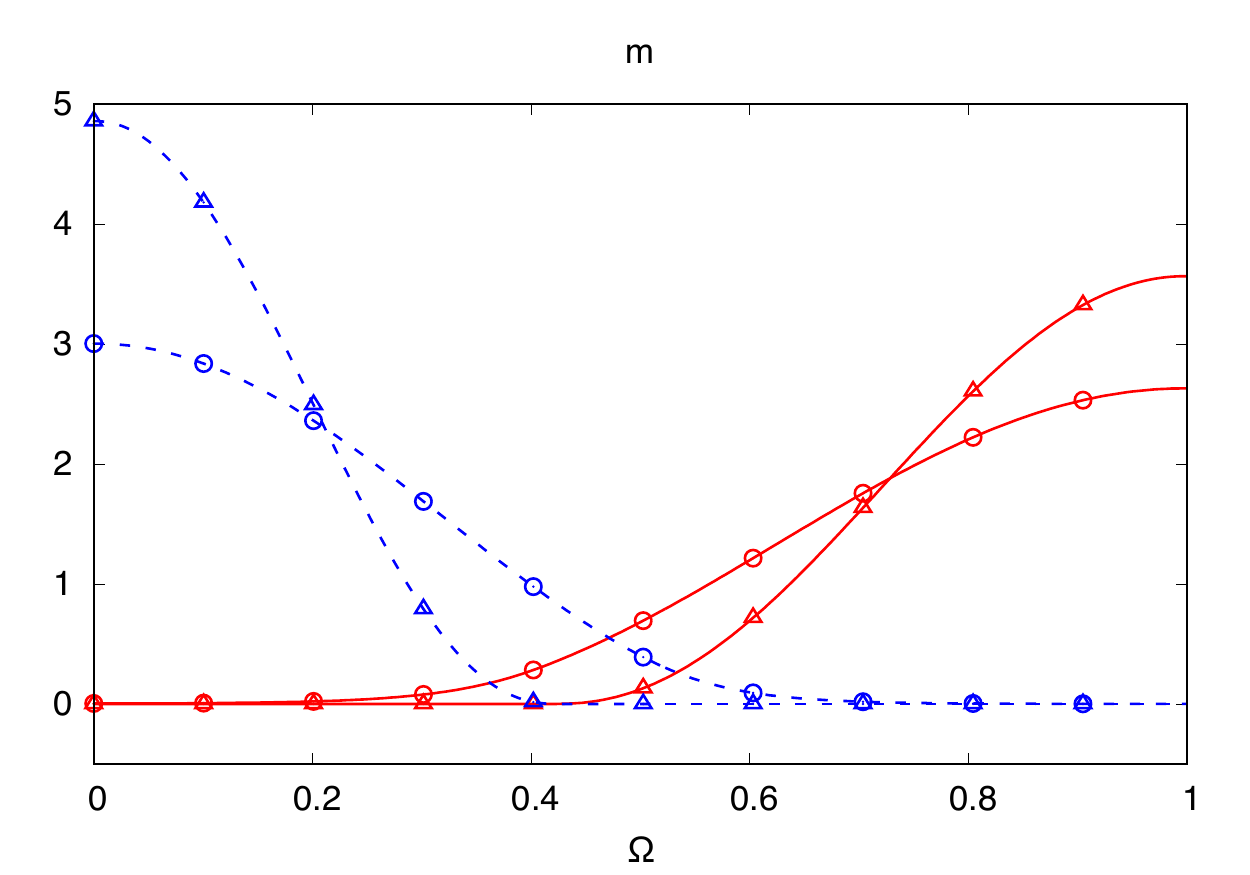}
     \includegraphics[width=6cm]{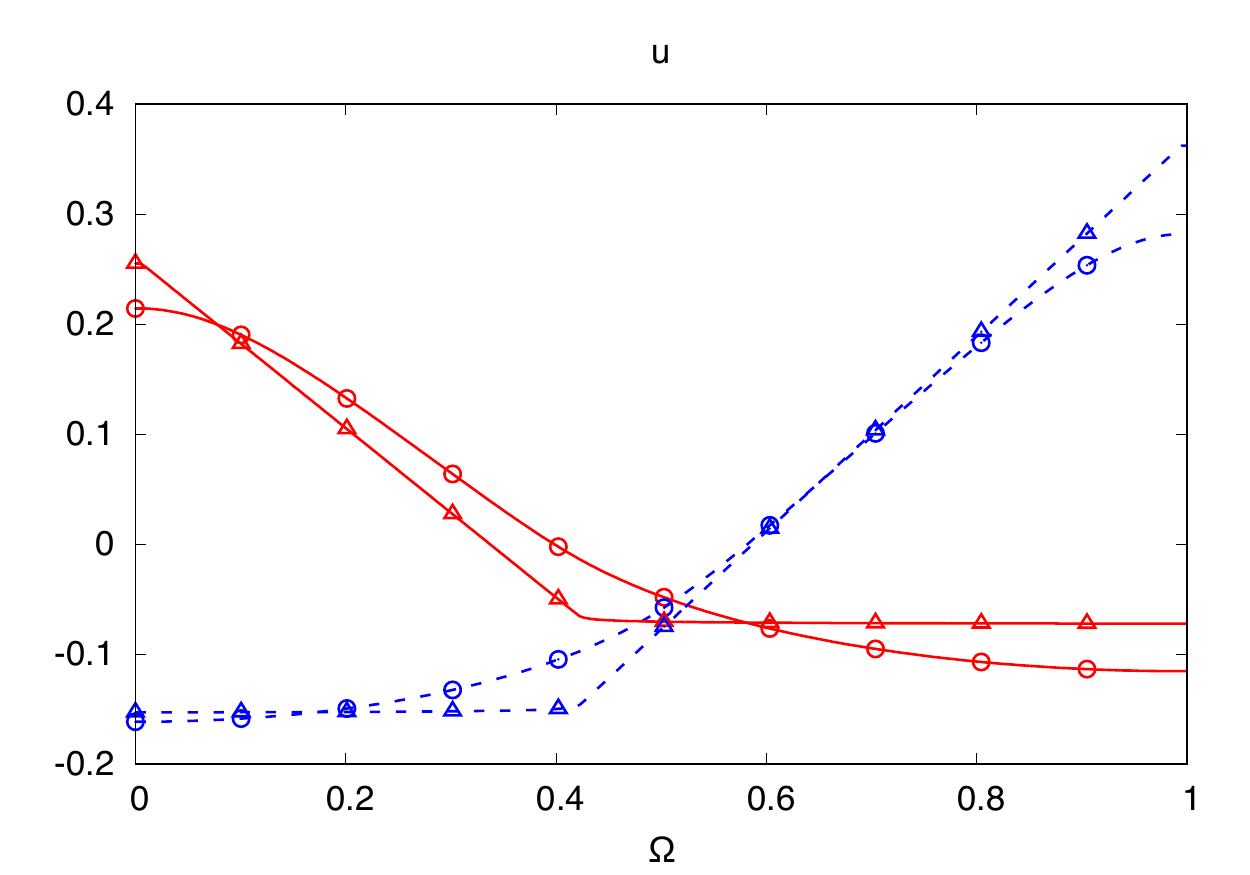}
\end{figure}

\begin{figure}
\centering
\caption{\footnotesize Another configuration, $\nu = 0.001$, with relative errors.}\label{fig_test1_2}
    \includegraphics[width=6cm]{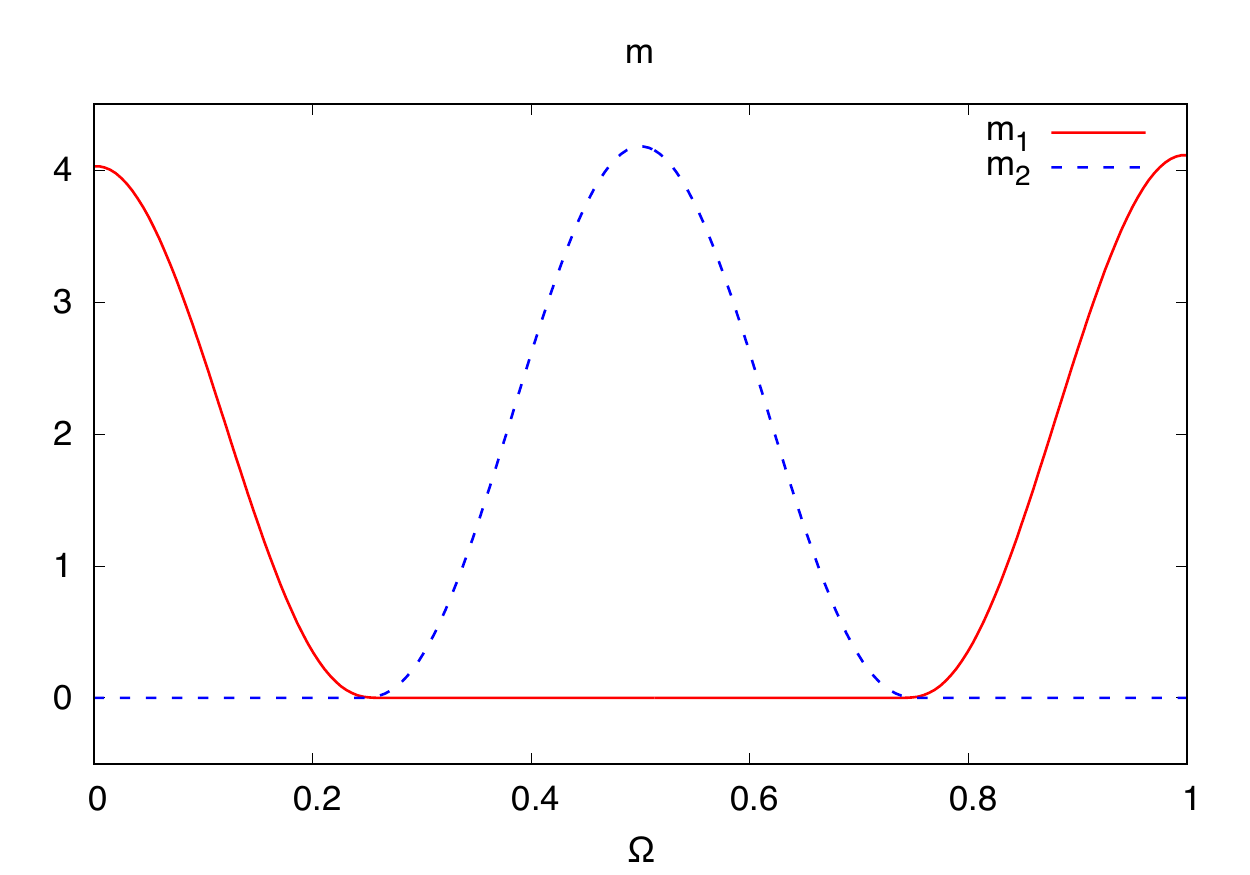}
    \includegraphics[width=6cm]{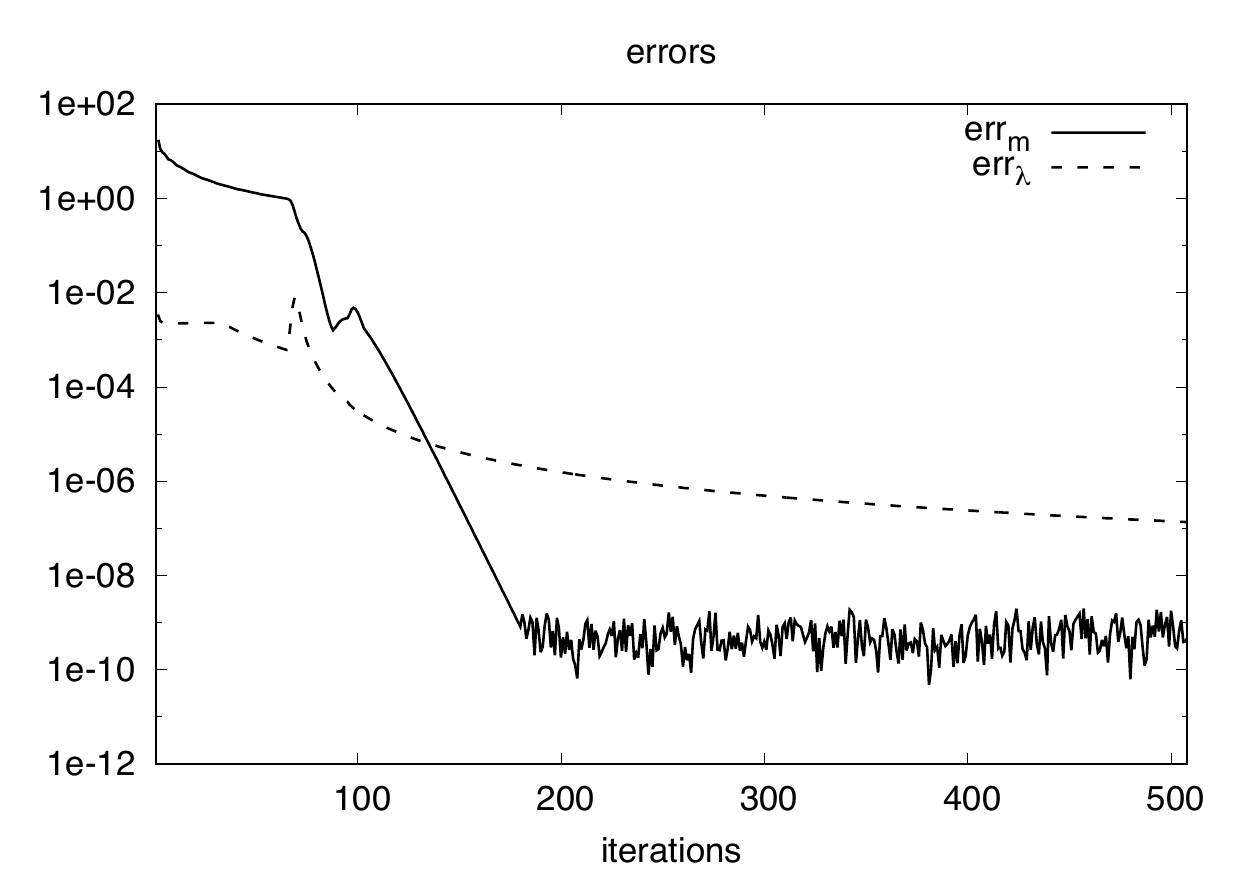}
\end{figure}

\begin{table}[ht]
\centering
\caption{The value of $\int m^1 m^2 \, dx$ versus $\nu$}\label{nu_vs_int}
\begin{tabular}{|C|C|}
\hline
\nu & h \sum_i m^1_{h,i} m^2_{h,i} \\ \hline \hline
0.05 & 0.09100195573  \\ \hline
0.01 & 0.00017126474  \\ \hline
0.005 & 0.00000811663  \\ \hline
0.0005 & 0.00000000068  \\ \hline
\hline
\end{tabular}
\end{table}

\smallskip {\bf Test 2.} In this test, we show how the ``family effect'' affects the behavior of the two populations, and compare the approximate solutions of \eqref{MFGstat} with local couplings $\overline{V}_\ell$ and $V_\ell$. In general, the presence of the family effect discourages segregation, and the two distributions appear to be a bit more ``mixed'' in this case, see Figure \ref{fig_test2}. Nevertheless, full segregation still occurs as $\nu$ approaches zero. Note that $V^k_\ell$ is positive where $m^k_h$ is close to zero, while $\overline{V}_\ell$ is proportional to $m^k_h$: what happens is that $\overline{V}_\ell$ is different from zero only in a (very) small region around the free boundary between $m_1$, $m_2$, that still is sufficient to trigger segregation if the viscosity is small.

\begin{figure}
\centering
\caption{\footnotesize The family effect. $\overline{V}_\ell$, left. $V_\ell$, right.}\label{fig_test2}
    \includegraphics[width=6cm]{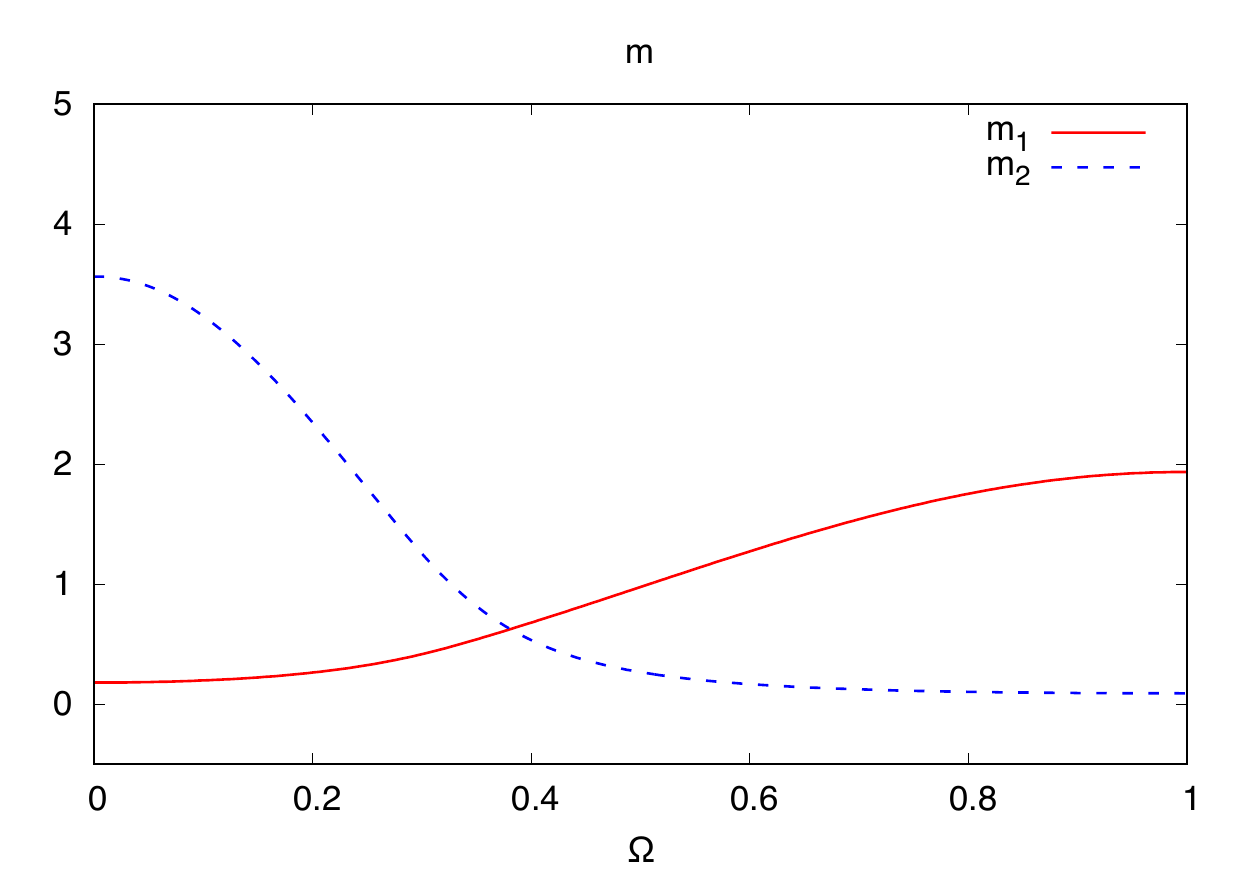}
    \includegraphics[width=6cm]{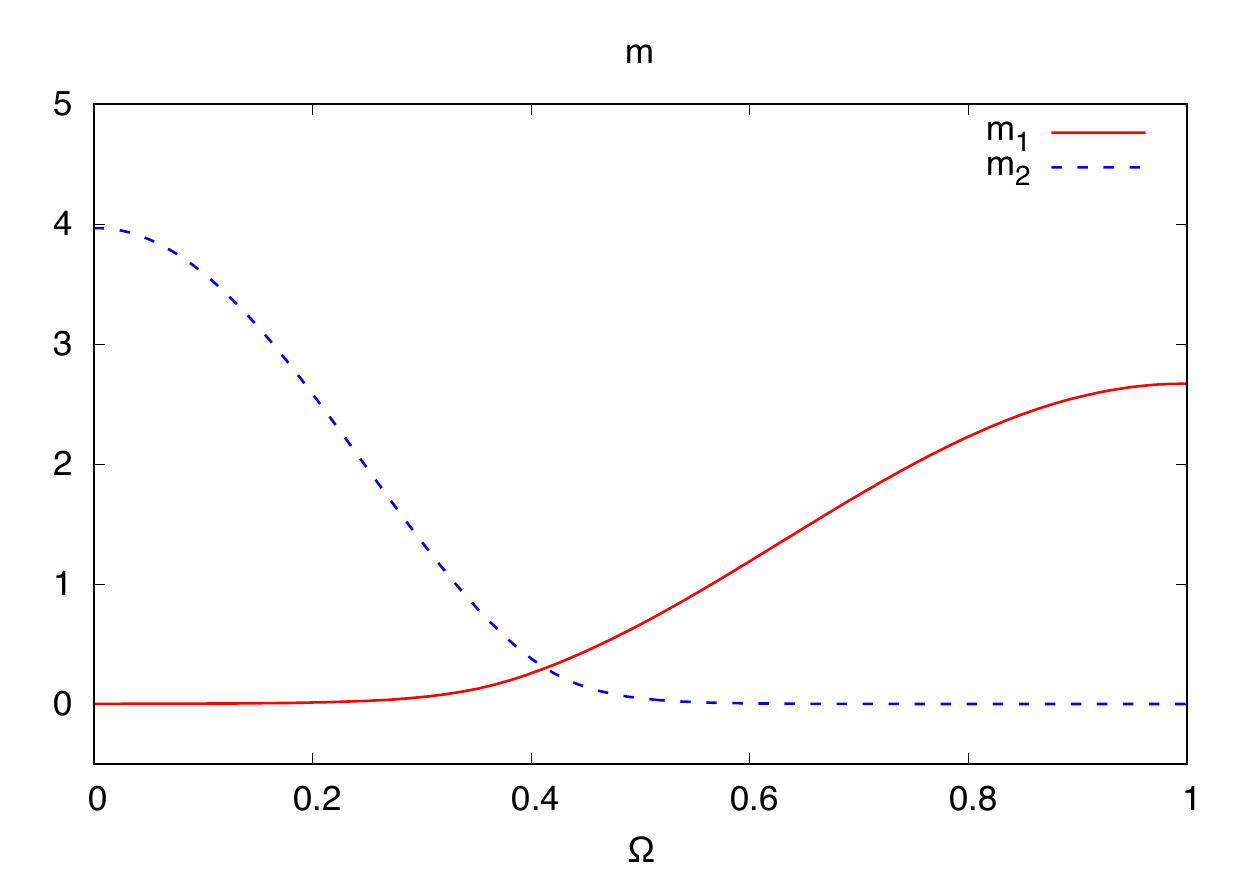}
\end{figure}

\smallskip {\bf Test 3.} In the previous tests, we used the local versions of the costs, namely we considered myopic players. Here, we show the results obtained considering the non-local versions of the cost functionals as in \eqref{Vk}, with kernel
\[
K(x,y) = \frac{1}{|[x-\delta, x+\delta] \cap \Omega|}\chi_{[x-\delta, x+\delta] \cap \Omega}(y), \quad \delta \in (0,1).
\]
In Figure \ref{fig_test3} the solutions $m_h$ are plotted; here, $\delta = 0.2$. If $a_1 = a_2 = 0.8$, players prefer regions of $\Omega$ with prevalent presence of their own population. In this case, one may observe that the set where both the distributions vanish as $\nu \to 0$ is an interval with non-empty interior; this is a consequence of the fact that the cost at position $x$ paid by a player depends on an entire neighbourhood of $x$. Nevertheless, if the happiness thresholds are sufficiently low (say, less than $0.5$, as in Figure \ref{fig_test3} (right)), the free boundary becomes a point, as in the local case $V_\ell$.

\begin{figure}
\centering
\caption{\footnotesize The non-local case: $m_1$ and $m_2$ in the subinterval $[0.35, 0.65]$. $a_1=a_2 = 0.8$, left. $a_1=a_2 = 0.3$, right.}\label{fig_test3}
    \includegraphics[width=6cm]{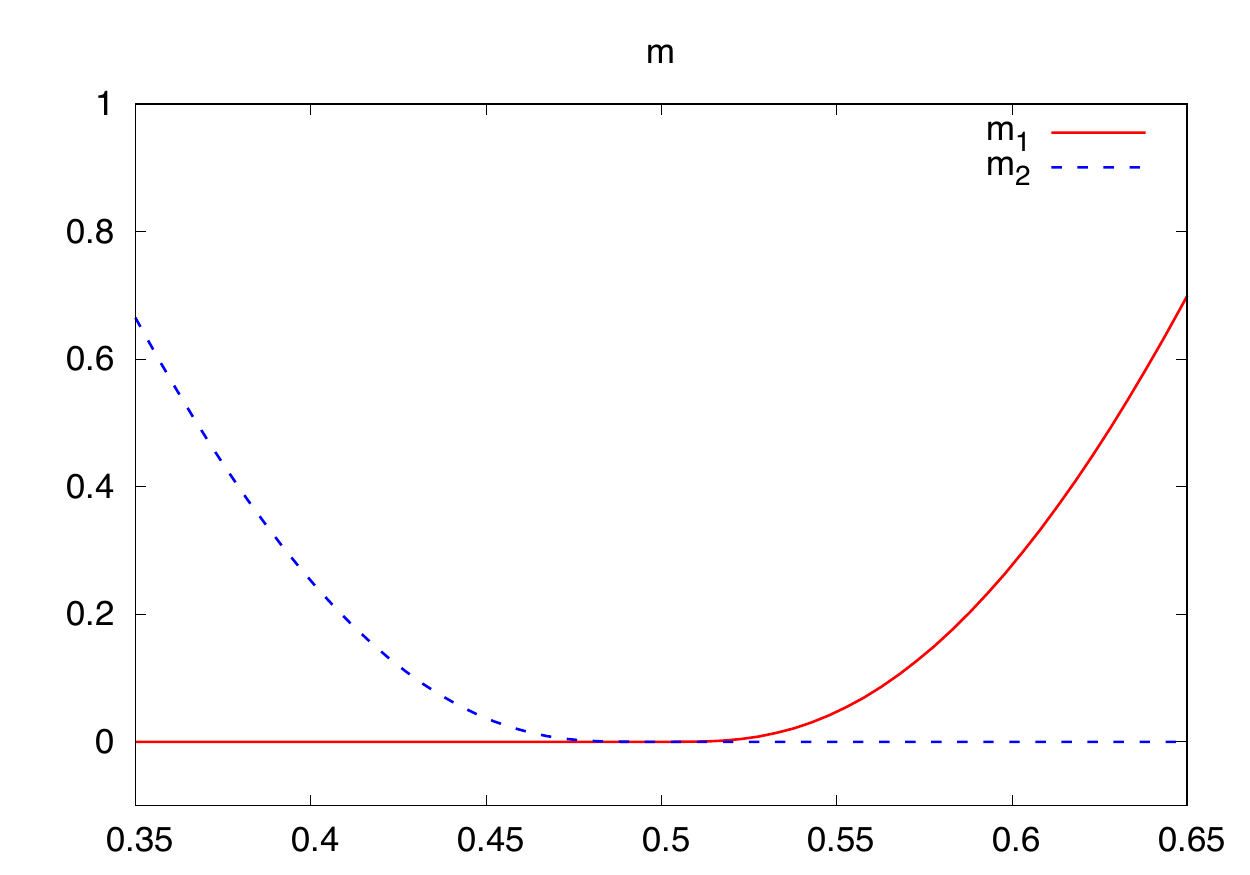}
    \includegraphics[width=6cm]{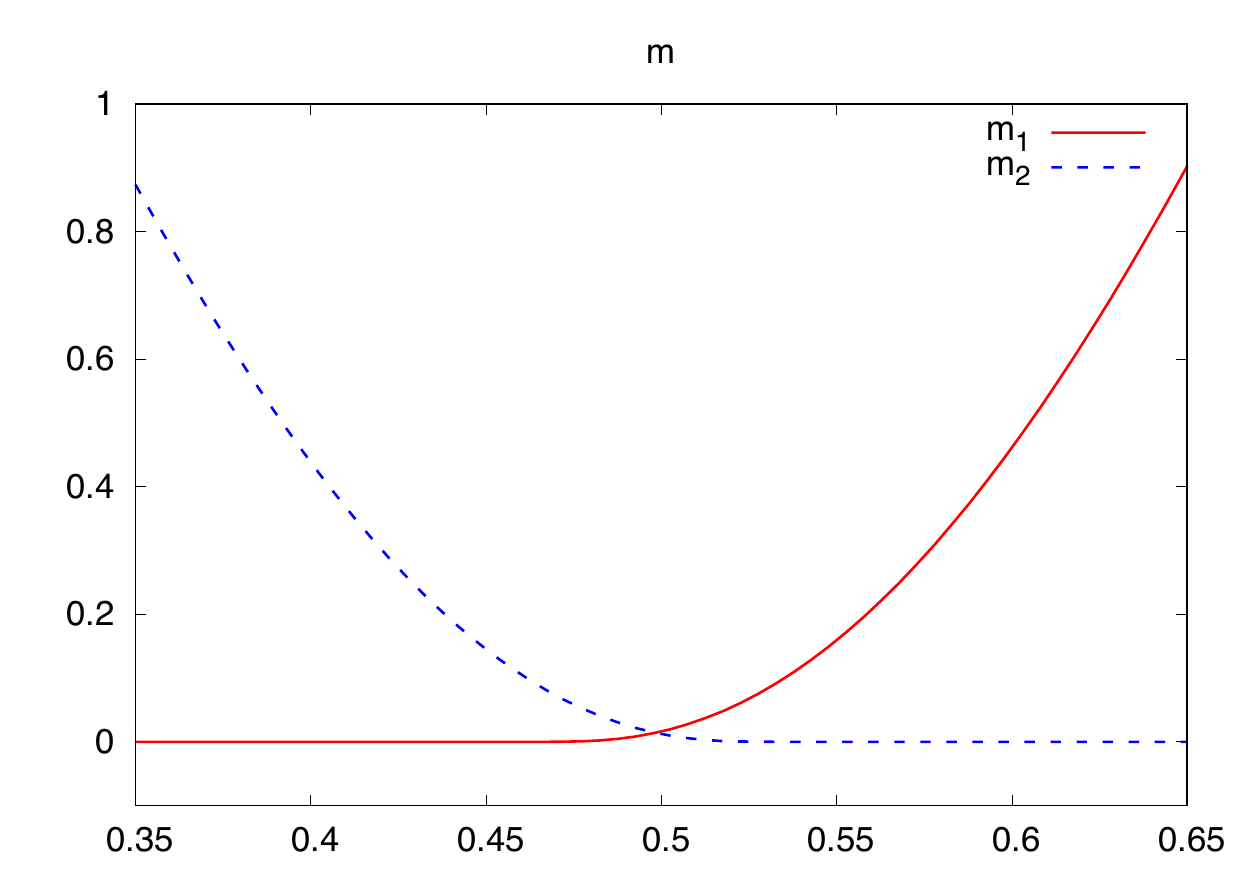}
\end{figure}

\smallskip {\bf Test 4.} In this test, we choose different parameters for the Hamiltonians. The value of $\gamma$ affects the shape of the distributions on their support, as shown in Figure \ref{fig_test4}. Still, different values of $\gamma$ produce segregation to the same extent.

\begin{figure}
\centering
\caption{\footnotesize The non-quadratic case. $\gamma = 8$, left. $\gamma = 4/3$, right.}\label{fig_test4}
    \includegraphics[width=6cm]{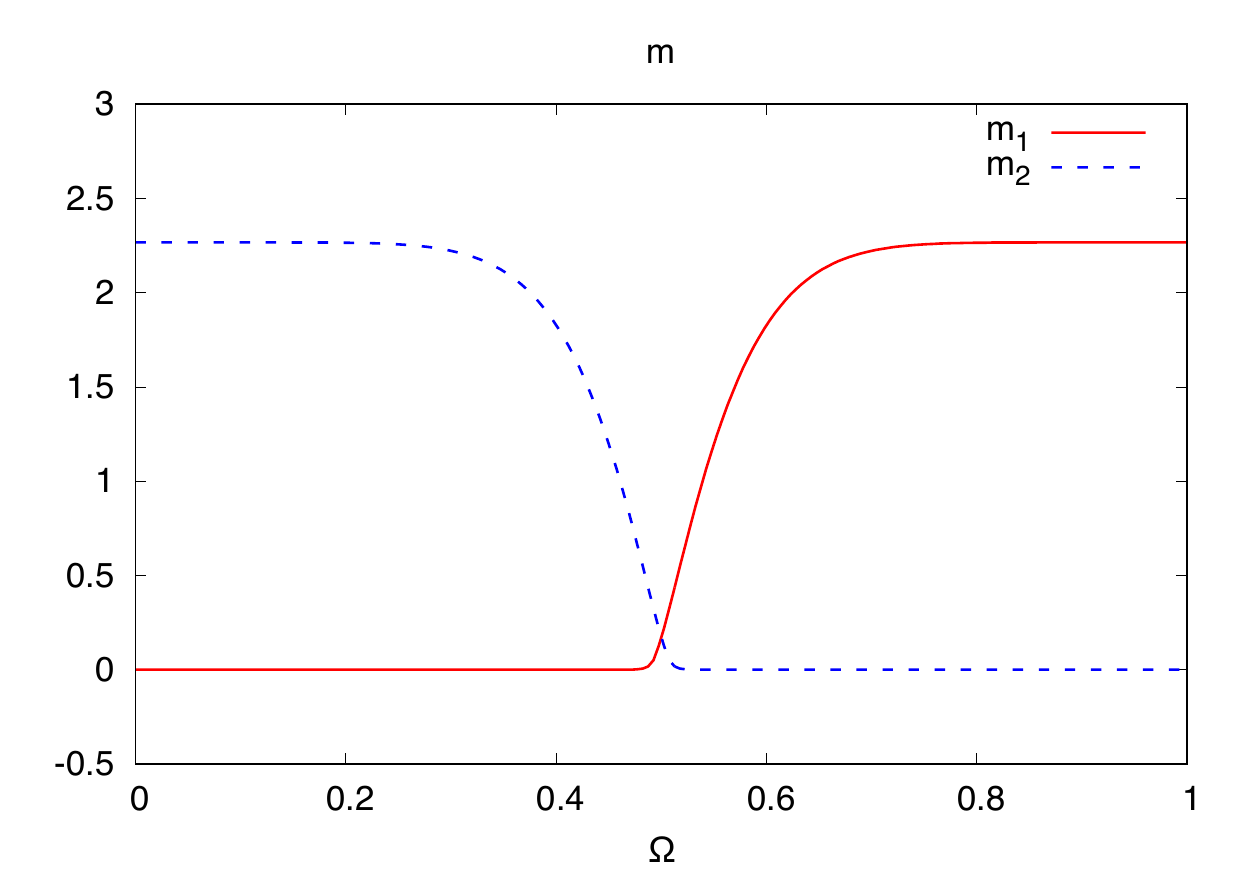}
    \includegraphics[width=6cm]{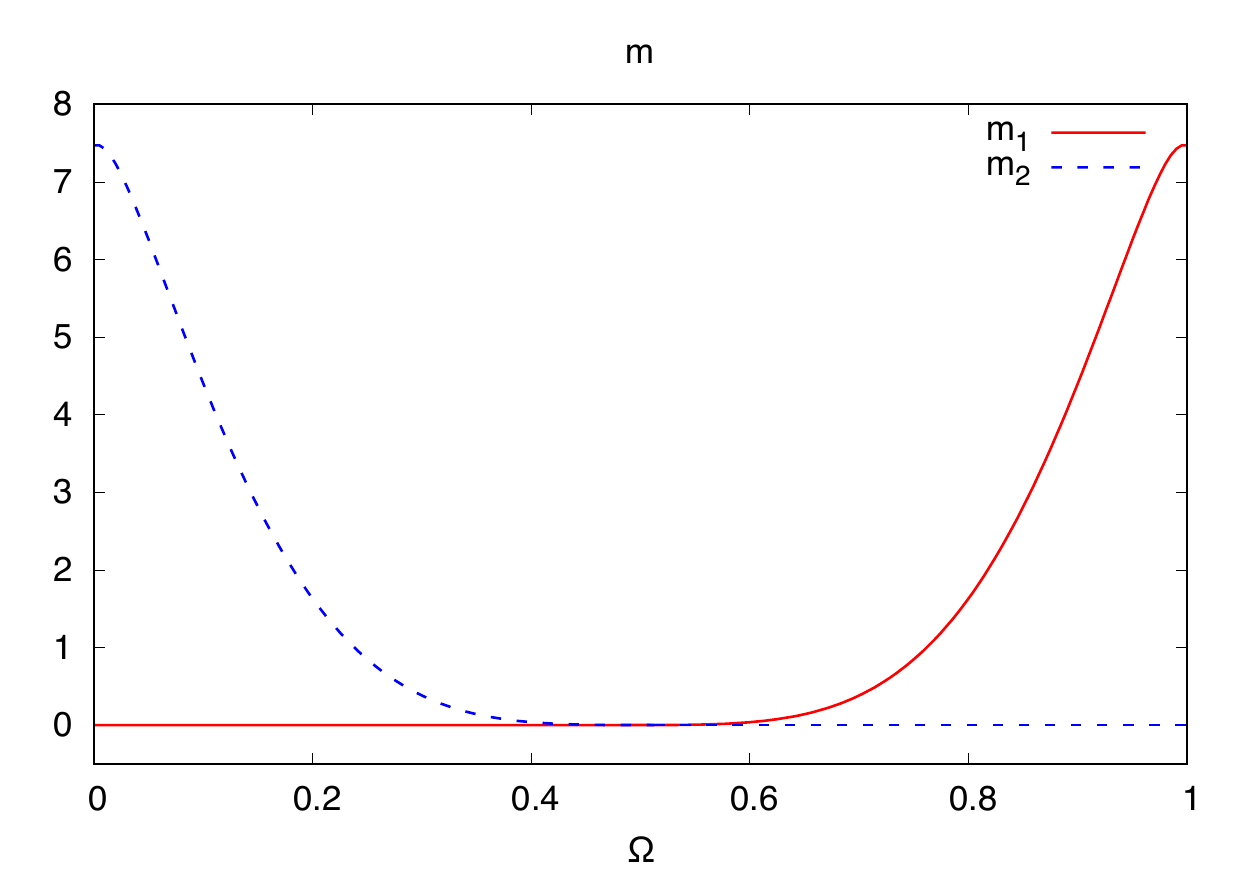}
\end{figure}

\begin{rem}\label{rem_CV}
Numerical simulations suggest the presence of a wide variety of solutions of \eqref{MFGstat} even in space dimension $d=1$. In Ref. \cite{CirantVerzini}, a similar system MFG is considered, where $\gamma = 2$ and $V^k(m_1, m_2)$ are just increasing functions of $m_{3-k}$; with respect to our models, segregation is even more encouraged, as players aim at avoiding the other population in any case. In their framework, some numerical phenomena arising here have been proven rigorously: existence of branches of solutions having one ore more critical points, and segregation as $\nu \to 0$, namely
\[
\int_\Omega m_1 m_2 \to 0.
\]
Moreover, in the vanishing viscosity limit, uniform bounds on $m$ are shown, indicating that concentration of the distribution is not likely to happen (therefore, anti-overcrowding terms in the costs as in Section \ref{family} might be unnecessary), and segregated configurations can be characterized by optimal partition problems. We believe that such features of \eqref{MFGstat} can be proven also for our Schelling models.
\end{rem}

\subsection{Evolutive   PDEs} 
Let us discuss the  numerical simulations of  some finite horizon problems.
\subsubsection{A one-dimensional case}
\label{sec:one-dimensional-case}
 Here, we choose $d = 1$, $\Omega = (-0.5,0.5)$ and the horizon $T=4$.
The parameter $\nu$ will take the two values $0.12$ and $0.045$.  
The value functions and the densities satisfy Neumann conditions at the two endpoints.
The  Hamiltonian is $H(x,p)= |p|^2$.
The terminal cost is $0$ and the coupling terms are of the form  $V^1[m_1,m_2](x)= V_\epsilon(m_1(x),m_2(x))$ and 
$V^2[m_1,m_2](x)= V_\epsilon(m_2(x),m_1(x))$, with 
\begin{displaymath}
          V_\epsilon(m, n)=
 \Psi_{-, \epsilon} \left( \frac {m} {m+n +\epsilon} - 0.7  \right)  +    \Psi_{+, \epsilon}(m+n- 8) 
\end{displaymath}
where
\begin{equation}\label{approx_pnpart}
 \Psi_{-, \epsilon} (y)=\left\{
   \begin{array}[c]{ll}
-y+ \frac \epsilon 2 (e^{\frac y \epsilon} -1)  \quad &\hbox{if } y\le 0  \\     
\frac \epsilon 2 (e^{-\frac y \epsilon} -1)  \quad &\hbox{if } y\ge 0
   \end{array}
\right.
\quad \hbox{and}\quad  \Psi_{+, \epsilon} (y)=\left\{
   \begin{array}[c]{ll}
\frac \epsilon 2 (e^{\frac y \epsilon} -1)  \quad &\hbox{if } y\le 0  \\     
y+\frac \epsilon 2 (e^{-\frac y \epsilon} -1)  \quad &\hbox{if } y\ge 0,
   \end{array}
\right.
\end{equation}
and  $\epsilon=10^{-5}$. The function $V_\epsilon$ is a regularized version of 
\begin{displaymath}
       V(m, n)= \left( \frac {m} {m+n} - 0.7  \right)^- + \left( m+n- 8\right)^+ .
\end{displaymath}
In this case,  the two populations are symmetric to each other. 
The first part of the coupling term stands for xenophobia: 
an agent located at $x$ pays a cost if at $x$, 
the proportion of agents of  its own type is less than $70\%$.
The second part models the aversion to overcrowded locations:
  an agent  located at $x$ pays a cost if the density of agents of both types at $x$ is greater than $4$. 
The initial densities are 
$m_{1,0}(x)= 3/4 + 1/2\chi_{[-1/2,-1/4]\cup [0,1/4]}(x)$ 
and $m_{2,0}(x)= 3/4 + 1/2\chi_{[-1/4,0]\cup [1/4,1/2]}(x)$.
Since the initial distributions are symmetric to each other 
and the population have symmetric characteristics, 
the distributions should remain symmetric for all times.  
\\
 The spatial grid step is $h = 1/50$ and the time step is $\Delta t =1/100$.
\\
For $\nu=0.12$, the evolution of the distributions of agents
 is displayed on Figure~\ref{fig_ya_test1}, which contains 
nine snapshots corresponding to different dates between $0$ and $T$. 
We easily see that the distributions of the two types of agents remain symmetric 
to each other. The distributions seem to keep oscillating  between two configurations
 in which the populations  are segregated and grouped in opposite sides of the domains.  
A possible explanation of this behavior may be as follows: in that rather particular situation when 
the two populations are symmetric to each other and
strongly xenophobic, a rather high level of noise makes it difficult 
to reach a global steady equilibrium.  We expect that there exists another solution 
which comes close to a steady equilibrium  for times not too close to $0$ and $T$ (see the next case with $\nu=0.045$), 
but this solution has not been selected by our numerical method.
\\
For $\nu=0.045$,  the evolution of the distributions is displayed on Figure~\ref{fig_ya_test2}.
Here again, the two distributions of agents remain symmetric to each other, but this time, we see that
the populations are very close to a steady equilibrium when $t$ is not too close to $0$ and $T$. 
The latter equilibrium is a configuration in which the two populations occupy disjoint subdomains.
\begin{figure}[htbp]
\centering
\caption{\footnotesize Evolution of $m_h$ for $\nu=0.12$: solid/red  (respectively dashed/blue) lines are used for  $m_1$, (respectively $m_2$).}
\label{fig_ya_test1}
    \includegraphics[width=14cm]{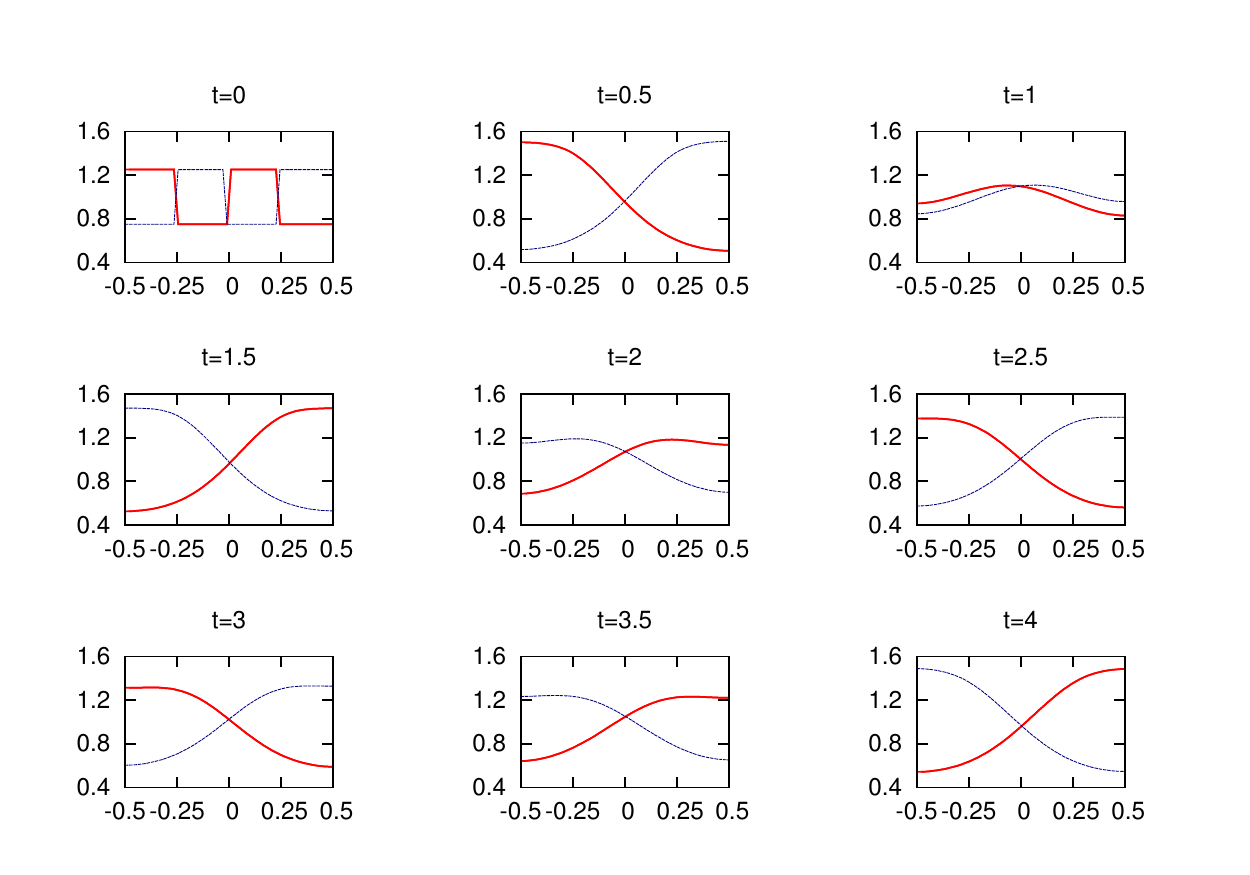}
\end{figure}

\begin{figure}[htbp]
\centering
\caption{\footnotesize Evolution of $m_h$ for $\nu=0.045$:  solid/red  (respectively dashed/blue) lines are used for  $m_1$, (respectively $m_2$).}\label{fig_ya_test2}
    \includegraphics[width=14cm]{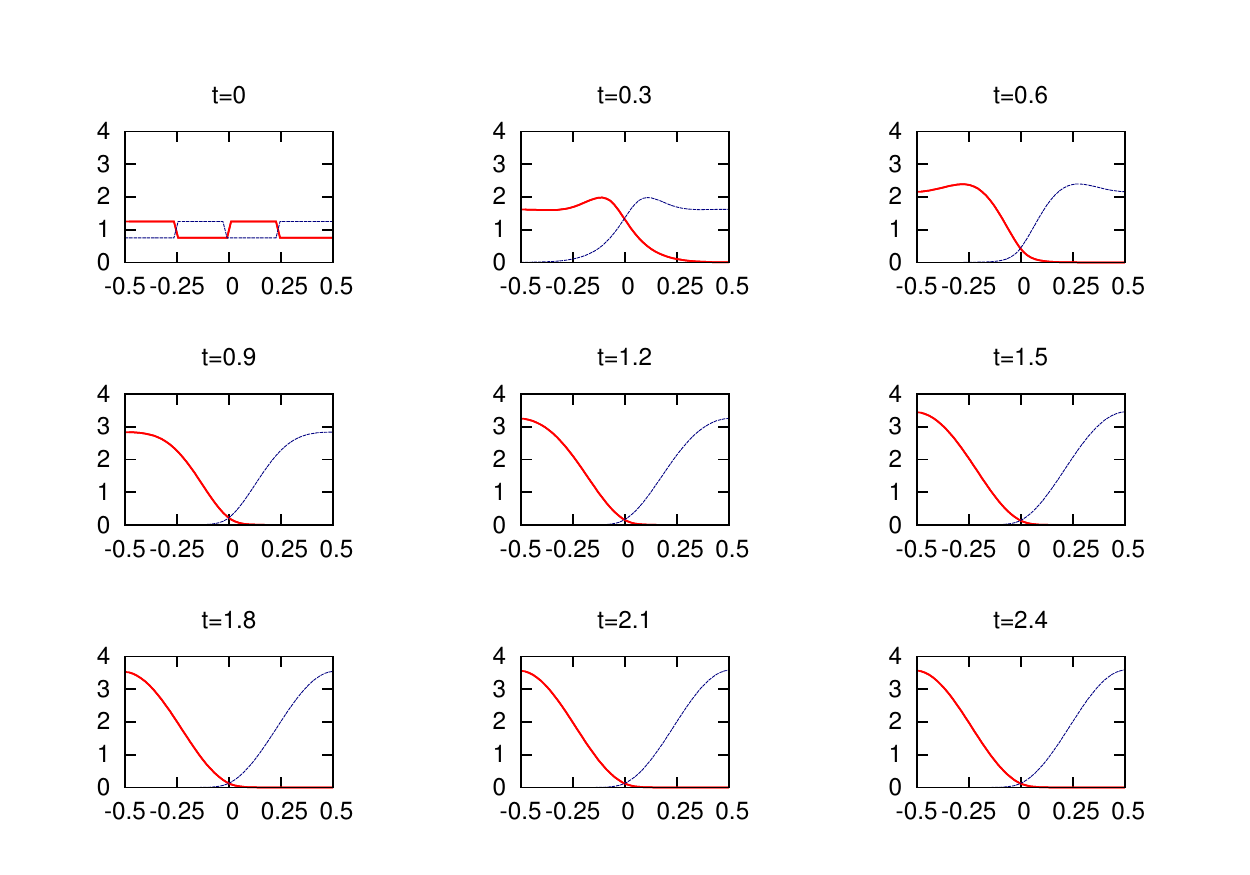}
\end{figure}

\subsubsection{Two bidimensional cases}\label{sec:two-dimensional-case}
\paragraph{Case a)}
Here the domain  $\Omega$ is obtained by removing a crossed-shaped  set from the unit square $(-0.5,0.5)^2$.
We consider two types agents bound to stay in $\overline \Omega$,  both 
with ``threshold of happiness'' $a_i$ below $1/2$.
More precisely, the model is as follows: the  Hamiltonians are $  H^1(x,p)=     H^2(x,p)=  |p|^2 $.
We take $\nu=0.038$; the value functions and the densities satisfy Neumann conditions at $\partial \Omega$.
\\
The terminal cost is $0$ and the coupling terms are of the form  
\begin{
displaymath}
  \begin{split}
    V^1_\epsilon [m_1, m_2](x)&= 2\Psi_{-, \epsilon} \left( \frac {m_1(x)} {m_1(x)+m_2(x) +\epsilon} - 0.5  \right) 
 +    \Psi_{+, \epsilon}(m_1(x)+m_2(x)- 8), \\
    V^2_\epsilon [m_1, m_2](x)&= 
 \Psi_{-, \epsilon} \left( \frac {m_1(x)} {m_1(x)+m_2(x) +\epsilon} - 0.4  \right)  +    \Psi_{+, \epsilon}(m_1(x)+m_2(x)- 8), 
  \end{split}
\end{
displaymath}
where $\Psi_{-, \epsilon}$ and $\Psi_{+, \epsilon}$ are defined in \S~\ref{sec:one-dimensional-case} and $\epsilon=10^{-5}$.
 These coupling terms are regularized versions of
\begin{displaymath}
  \begin{split}
    V^1[m_1, m_2](x)&=
 2 \left( \frac {m_1(x)} {m_1(x)+m_2(x) +\epsilon} - 0.5  \right)^- + \left( m_1(x)+m_2(x)- 8\right)^+ ,            \\
    V^2 [m_1, m_2](x)&=  \left( \frac {m_2(x)} {m_1(x)+m_2(x)+\epsilon} - 0.4  \right)^- + \left( m_1(x)+m_2(x)- 8\right)^+ .
  \end{split}
\end{displaymath}
Note that the first population is less tolerant  than the second one.\\
The agents of the first (respectively second) type are initially uniformly distributed in the top half part (right half part)  of the domain, with a density of $2$.
Therefore, in the top-right corner of the domain, the two populations are initially mixed  and 
the less tolerant  agents are in an uncomfortable state.
Moreover, the cost for staying in  that part of the domain is higher for the first population 
 of agents (by the factor $2$ multiplying the term $(...)^-$).
 \\
\begin{figure}[htbp]
\centering
\caption{\footnotesize Evolution of $m_h$ for $\nu=0.038$:  red   (respectively blue) colors are used for  $m_1$, (respectively $m_2$).}\label{fig_ya_test4}
    \includegraphics[width=13cm]{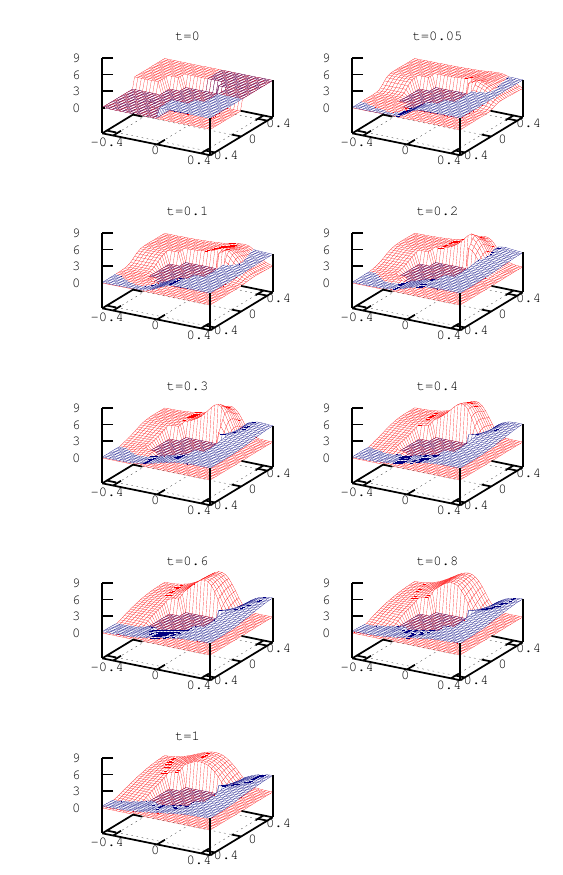}
\end{figure}
In the  simulation, the spatial grid step is $1/64$ and the time step is $1/100$. The evolution of the distributions is displayed on Figure~\ref{fig_ya_test4}:
we see that  the first population 
 leaves the top-right corner and moves toward to the top-left corner of the domain.
The second population, which is 
more tolerant, remains in the top-right corner, evolves in a slower manner, and tends to occupy a larger part of 
the domain than the first one. \\
Note the Schelling's phenomenon: segregation occurs even if both thresholds $a_1=0.5, a_2=0.4$ are not xenophobic.


\paragraph{Case b)}
Here, we consider a case when the two types of agents move in order to reach two different targets: the strategy of the agents consists of reaching the targets while avoiding
 the agents of the other population. Hence, the dynamics of the agents is not only motivated by xenophobia.\\ 
The domain  is the unit square 
 $\Omega = (0,1)^2$ and  the horizon is $T=1$.
\\
The agents of the first (respectively second) type are initially distributed in the top-left
(respectively bottom-left) corner of the domain, but are attracted 
toward the bottom-right  (respectively top-right) corner to avoid the running costs.
Therefore,  the strategy of the agents will be
obtained as a trade-off between two opposite tendencies:
on the one hand, the agents would like to quickly 
reach the opposite corner, taking paths  which cross each other, 
but on the other hand the two populations try to avoid each other.
\\
More precisely, the model is as follows: the  Hamiltonians are
\begin{displaymath}
  \begin{split}
    H^1(x,p)&=   |p|^2 -   1.4   \chi_{  [0,0.7]\times [0.2,1]}(x),\\
    H^2(x,p)&=   |p|^2 -   1.4   \chi_{  [0,0.7]\times [0,0.8]}(x)),
  \end{split}
\end{displaymath}
which means in particular that the first (respectively second) 
type of agents is attracted to the rectangle $[0.7,1]\times [0,0.2]$
  (respectively $[0.7,1]\times [0.8, 1]$).
We take $\nu=0.03$; the value functions and the densities satisfy Neumann conditions at $\partial \Omega$.
\\
The terminal cost is $0$. The coupling terms are 
 \begin{displaymath}
   \begin{split}
     V^1[m_1, m_2](x)&=
  2 \left( \frac {m_1(x)} {m_1(x)+m_2(x) +\epsilon} - 0.8  \right)^- + \left( m_1(x)+m_2(x)- 8\right)^+ ,            \\
     V^2 [m_1, m_2](x)&=  \left( \frac {m_2(x)} {m_1(x)+m_2(x)+\epsilon} - 0.6  \right)^- + \left( m_1(x)+m_2(x)- 8\right)^+ .
   \end{split}
 \end{displaymath}
The first population is more xenophobic than the second one.
The initial distributions of the agents are given by
\begin{displaymath}
  \begin{split}
    m_{1,0}(x)&= 4  \chi_{(0,0,2)\times (0.6,1)}+ 0.02,\\
    m_{2,0}(x)&= 4  \chi_{(0,0,2)\times (0,0.4)}+ 0.02.
  \end{split}
\end{displaymath}
In the  simulation, the spatial grid step is $1/64$ and the time step is $1/100$.\\
The evolution of the distributions is displayed on Figure~\ref{fig_ya_test3}:
we see that in the beginning  (before $t=0.2$), a significant part of the 
 first population (the more xenophobic agents) quickly moves 
 to the opposite corner:  even if those agents pay an important cost for quickly moving  
to the opposite corner, this cost is compensated by   their quickly reaching a location 
where there are no agents of type 2. By contrast, for $t\le 0.2$ 
the second population is more uniformly distributed. At time $t=0.2$, the first population is split 
into  two groups: the first group has almost reached the desired corner,
 whereas the second group has not moved.
Next, for $0.2\le t\le 0.6$, this latter  group of agents of the  
 first type still does not move,  while the whole second population moves to its favorite corner,
 occupying the center of the domain. Indeed, since the density of the
 agents of the second type in the middle of the domain 
has become too important, the agents of the first type prefer waiting rather than 
meeting them.  
 At $t=0.6$, most of the second population has reached the desired corner,
 and the first population can finish crossing the domain.
 
\begin{figure}[htbp]
\centering
\caption{\footnotesize Evolution of $m_h$ for $\nu=0.03$:  red   (respectively blue) colors are used for  $m_1$, (respectively $m_2$).}\label{fig_ya_test3}
    \includegraphics[width=13cm]{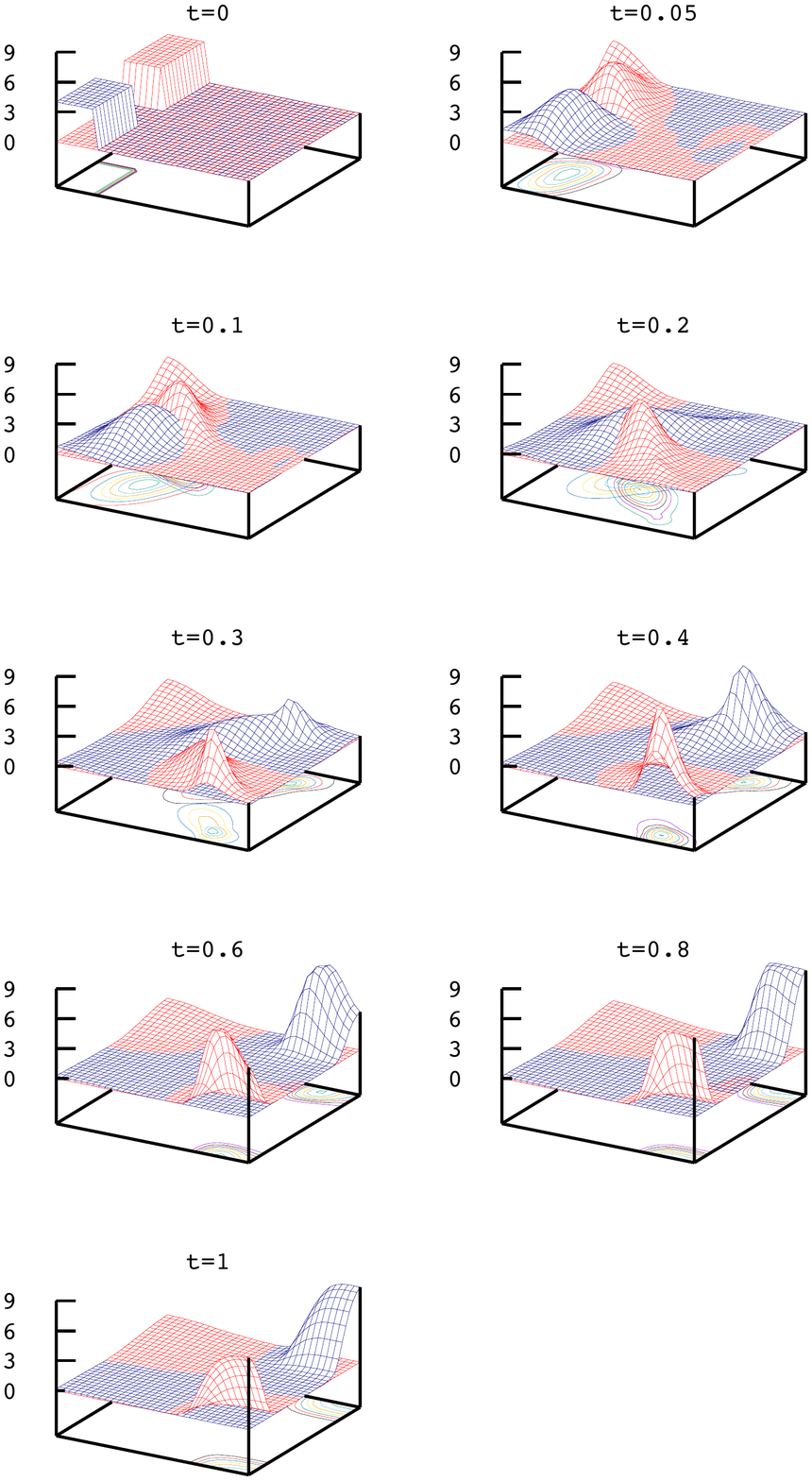}
\end{figure}


\subsubsection*{\bf Acknowledgements}
  The first author  was partially funded  by the ANR 
projects ANR-12-MONU-0013 and ANR-12-BS01-0008-01. The second author is partially supported by the research project of the University of Padova ``Mean-Field Games and Nonlinear PDEs''.
The second and third authors  are members of the Gruppo Nazionale per l'Analisi Matematica, la Probabilit{\`a} e le loro Applicazioni (GNAMPA) of the Istituto Nazionale di Alta Matematica (INdAM), and are partially supported by the GNAMPA project ``Fenomeni di segregazione in sistemi stazionari di tipo Mean Field Games a pi\`u popolazioni''.

\bibliographystyle{plain}

\medskip
\begin{flushright}

\noindent \verb"achdou@math.jussieu.fr"\\
UFR Math{\'e}matiques, Universit{\'e} Paris Diderot, \\
Case 7012, 75251 Paris Cedex 05, France, and\\
Laboratoire Jacques-Louis Lions, \\
Universit{\'e} Paris 6, 75252 Paris Cedex 05

\smallskip

\noindent \verb"bardi@math.unipd.it"\\
Dipartimento di Matematica, Universit{\`a} di Padova\\
via Trieste, 63, I-35121 Padova, Italy

\smallskip

\noindent \verb"marco.cirant@unimi.it"\\
Dipartimento di Matematica, Universit\`a di Milano\\
via Cesare Saldini 50, 20133 Milano, Italy

\end{flushright}

\end{document}